\documentclass{llncs}

\usepackage{a4wide}

\usepackage{subfigure}
\usepackage{xspace}
\usepackage{graphicx}
\usepackage{amsmath}
\usepackage{amssymb}
\usepackage{comment}

\def\S{\ensuremath{\mathcal{S}}\xspace}
\def\aut{\ensuremath{\mathrm{Aut}}}

\def\HM{\mathcal{M}} 
\def\HC{\mathcal{C}} 
\def\HG{\mathcal{H}} 
\def\bx{\mathbf{x}}
\def\by{\mathbf{y}}

 \title{Bijections for simple and
  double Hurwitz numbers} \author{E. Duchi, D. Poulalhon and
  G. Schaeffer} \institute{LIAFA, Universit\'e Paris Diderot, and LIX,
  CNRS et \'Ecole Polytechnique }
\begin{document}
\maketitle
\begin{abstract}
We give a bijective proof of Hurwitz formula for the number of simple
branched coverings of the sphere by itself.  Our approach extends to
double Hurwitz numbers and yields new properties for them.  In
particular we prove for double Hurwitz numbers a conjecture of
Kazarian and Zvonkine, and we give an expression that in a sense
interpolates between two celebrated polynomiality properties:
polynomiality in chambers for double Hurwitz numbers, and a new analog
for almost simple genus 0 Hurwitz numbers of the polynomiality up to
normalization of simple Hurwitz numbers of genus $g$.  Some
probabilistic implications of our results for random branched
coverings are briefly discussed in conclusion.
\end{abstract} 

\section{Introduction}
Hurwitz branched covering counting problem consists in determining the
number of inequivalent $d$-sheet branched coverings of the 2-sphere
$\S_0$ by a connected genus $g$ closed surface
$\S_g$, with prescribed ramification types over some fixed
set of critical points, most of which are simple.
In the last decade, the topic has attracted renewed interest following
the work of Okounkov and Pandharipande \cite{okounkov2001gromov}
showing how Hurwitz numbers could be used to derive an alternative
proof of Konsevitch theorem (Witten conjecture, see also
\cite{Kazarian-Lando}), and that of Ekedahl, Lando, Shapiro and
Vainshtein \cite{ekedahl2001hurwitz}, revealing a  tight
relation between Hurwitz numbers and intersection theory of the moduli
spaces of curves now known as ELSV formula (see \emph{e.g.}
\cite{dunin-barkowski13} and references therein).

Of particular interest are the \emph{double Hurwitz numbers}
$h_g(\mu,\nu)$, indexed by a non-negative integer $g$ and two
partitions $\mu$ and $\nu$ of $d$: they count $d$-sheet branched
coverings of $\S_0$ by $\S_g$ with $r+2$ fixed
ramified points, all of which are simple, except for the last two
which have ramification type $\mu$ and $\nu$ respectively (each
covering $f$ is counted with a weight $1/{\aut(f)}$ and $r=m+n-2+2g$,
$m$ and $n$ being the respective number of parts of $\mu$ and
$\nu$). A variety of approaches have been used to study these double
Hurwitz numbers, or their specialization to the simple Hurwitz numbers
$h_g(\mu)=h_g(\mu,1^d)$: cut-and-join equations
\cite{goulden1997transitive,cavalieri10}, the ELSV formula
\cite{ekedahl2001hurwitz,Goulden-Jackson-Vakil,SSV}, character
theoretic or infinite wedge approaches leading to integrable
hierachies
\cite{okounkov2000toda,goulden-KP,Johnson,natanzon,dunin-barkowski13},
matrix integrals \cite{borot2011matrix}, the topological recursion
\cite{topo-rec} and indirect bijective enumeration
\cite{bousquet2000enumeration}.
\smallskip

In this article, we introduce new combinatorial structures,
\emph{Hurwitz mobiles}, and a non-trivial one-to-one correspondence
between branched coverings and these objects. Hurwitz mobiles are
tree-like structures that are in some cases much easier to enumerate
than branched coverings or any of their known combinatorial avatars
(ribbon graphs, constellations, factorizations into transpositions, or
tropical diagrams...), and we use this fact to derive results of
different types:
\begin{description}
\item[a.] A 5 page self-contained bijective proof of Hurwitz formula
  for genus 0 simple Hurwitz numbers (modulo standard 
  results about trees).
\item[b.] For any fixed partition $\nu$, an analog for genus 0 almost
  simple Hurwitz numbers $h_0(\mu,\nu1^{d-|\nu|})$ of the
  polynomiality property of positive genus simple Hurwitz numbers
  $h_g(\mu,1^d)$ (Corollary~\ref{cor:pol}).
\item[c.] A simple expression as a sum of explicit positive monoms
  indexed by trees for the polynomials giving double Hurwitz numbers
  in chambers (Corollary~\ref{cor:chambers}), and related results.
\item[d.] A proof of a conjecture of Kazarian and Zvonkine about the
  dependency in $d$ of double Hurwitz numbers
  $h_0(\mu1^{d-|\mu|},\nu1^{d-|\nu|})$ when $\mu$ and $\nu$ are fixed,
  and new explicit formulas for these numbers with a remarkable
  positivity property (Theorem~\ref{thm:withfixpoints}).
\end{description}
Our main correspondence directly extends to higher genus and we
believe that \textbf{a.} could be adapted to give a common
generalization of \textbf{b.}  and the polynomiality properties of
higher genus simple Hurwitz numbers, but we have not done this
yet. Along with \textbf{c.}  we can rederive the chamber structure of
Hurwitz numbers (in particular the so-called resonances have a nice
interpretation in our setting, as well as Shadrin, Shapiro,
Vainstein's recurrence relation for wall crossing formulas \cite{SSV})
but we feel that this is not so interesting even from a combinatorial
perspective because, as shown by Cavalieri \emph{et al}
\cite{cavalieri10}, these results follow from the direct combinatorial
interpretation of the cut and join equation which is more elementary
than our approach. Regarding \textbf{d.}, we prove that for any two
fixed partitions $\mu$ and $\nu$, with $|\mu|\geq|\nu|$ (assumed for
the moment without parts equal to one for simplicity), the double
Hurwitz numbers $h_0(\mu1^{d-|\mu|},\nu1^{d-|\nu|})$ can be expanded
after proper normalization as polynomials in $d$:\\[-.5em]
\begin{equation}\label{for:Kazarian-Zvonkine}
\frac{h_0(\mu1^{d-|\mu|},\nu1^{d-|\nu|})}{(m+n+2d-|\mu|-|\nu|-2)!}=
\frac{d^{d+m+n-3}}{d!}\frac{(d)_{|\mu|}}{d^{|\mu|}}\frac{q_{\mu,\nu}(1/d)}
{\aut(\mu)\aut(\nu)}
\end{equation}
where $(d)_k=d(d-1)\cdots(d-k+1)$ and $q_{\mu,\nu}(z)$ is a polynomial
of degree $|\nu|$, as conjectured by Kazarian and Zvonkine. Moreover,
we give a combinatorial description of the coefficients of the
polynomials $q_{\mu,\nu}(z)$ in Theorem~\ref{thm:withfixpoints}, and
we show their dependancy in the parts of $\mu$ and $\nu$.  For
instance, for any $\alpha\geq\beta\geq2$, we prove\\[-1.5em]
\begin{eqnarray}\label{for:explicit}
\frac{h_0(\alpha1^{d-\alpha},\beta1^{d-\beta})}{(2d-\alpha-\beta)!}=\frac{d^{d-1}}{d!}\frac{\alpha^\alpha}{\alpha!}\frac{\beta^\beta}{\beta!}
\left(\frac{(d)_{\alpha+\beta}}{d^{\alpha+\beta}}
+\frac1d\cdot\sum_{\ell=1}^\beta\frac{(d)_{\alpha+\beta-\ell}}{d^{\alpha+\beta-\ell}}\frac{(\alpha)_{\ell}}{\alpha^{\ell}}\frac{(\beta)_{\ell}}{\beta^{\ell}}(\alpha+\beta-\ell)\right). \label{eq:ex}
\end{eqnarray}\\[-1em]
The cases $\{\alpha,\beta\}\subset\{2,3\}$ of this formula were
previously published in \cite{zvonkine2005enumeration,Kazarian} and an
algorithm to compute the coefficients for fixed $\alpha$ and $\beta$
is given in \cite{Kazarian} together with explicit formulas for larger
values of $\alpha$ and $\beta$ but to the best of our knowledge the
closed form above is new. In particular our formulas have the
remarkable property that that they involve summation of positive
contributions so that they are cancellation free (as opposed
\emph{e.g} to formulas in \cite{zvonkine2005enumeration} or
\cite{natanzon}).

Our polynomiality result for genus 0 almost-simple Hurwitz numbers is
in fact a further generalization of the result above: for each
partition $\nu$ we prove combinatorially that there exist 
polynomials $q_{\lambda,\kappa}^\nu$ in $m$ 
such that for all partitions $\mu=(\mu_1,\ldots,\mu_m)$ of an integer
$d\geq|\nu|$,\\[-.5em]
\begin{equation}\label{for:sym}
\frac{ h_0(\mu,\nu1^{d-|\nu|})}{
(m+n+d-|\nu|-2)!}
=\frac{d^{d-2-|\nu|}}{\aut(\mu)}\prod_{i=1}^m\frac{\mu_i^{\mu_i}}{\mu_i!}
\sum_{\begin{subarray}c \lambda,\kappa\mid |\lambda|+|\kappa|<|\nu|\end{subarray}}
q_{\lambda,\kappa}^\nu(m)\cdot\mathbf{m}_{\lambda;\kappa}(\mu_1,\ldots,\mu_m)
\end{equation}\\[-.5em]
where $\mathbf{m}_{\lambda;\kappa}$ denote the monomial symmetric
Laurent polynomial of shape $(\lambda;\kappa)$.  In particular for
$\nu=1$, the summation restricts to $\lambda=\kappa=\varepsilon$ the
empty partition and
$q^1_{\varepsilon,\varepsilon}=\mathbf{m}_{\varepsilon,\varepsilon}=1$
so that Formula~(\ref{for:sym}) is Hurwitz formula. More generally, we
derive the above structural result from an explicit expression of
double Hurwitz numbers as a finite sum of positive contributions
indexed by simple combinatorial structures. Again the existence of
formulas in terms of symmetric functions in the parts had been
observed for small values of $\nu$ by Kazarian \cite{Kazarian}.  Our
main interpretation of Hurwitz numbers, Theorem~\ref{thm:main-planar},
can be understood as a combinatorial interpolation between these
various polynomiality results.

To obtain our main correspondence, we adapt to Hurwitz' problem an
approach that has been developed in the last 15 years in the
combinatorial study of planar maps
\cite{PhD-Schaeffer,chassaing2004random,Bernardi-Fusy,Bernardi-Chapuy,Albenque-Poulalhon},
and that has allowed in particular to show that properly rescaled
large random planar maps admit a non trivial continuum limit, the
Brownian map \cite{gall2011uniqueness,miermont2011brownian}.  In this
context, two main strategies have emerged, based on the one hand on
minimal orientations and so-called \emph{blossoming trees} (see
\cite{Albenque-Poulalhon}), and on the other hand on geodesic labeling
and so-called \emph{mobiles} (see
\cite{chassaing2004random,BDFG04,Bernardi-Fusy,Bernardi-Chapuy}).  In
\cite{DuPoSc14} we have shown that the first approach could be
extended to derive a first bijective proof of Hurwitz' explicit
formula for genus 0 simple Hurwitz numbers. However this first
approach does not extend to higher genus and becomes more intricate in
the case of double Hurwitz numbers. The present paper builds on the
second strategy and in particular on Bouttier, Di Francesco, and
Guitter's bijection between bipartite maps and mobiles \cite{BDFG04}.
Refining and adapting this approach yields a different bijective proof
of Hurwitz formula, that extends nicely to double Hurwitz numbers. As
discussed in the conclusion we hope that theses results could lead, as
in the case of planar maps, to a better understanding of the
combinatorial geometry of branched coverings, and in particular to a
proof that a properly chosen model of rescaled random branched
coverings converges to the Brownian map.

\newpage

\section{Statement of the main bijection}
\subsection{Hurwitz numbers and galaxies}\label{sec:galax}

\begin{figure}[t]
\begin{center}
\includegraphics[scale=.5]{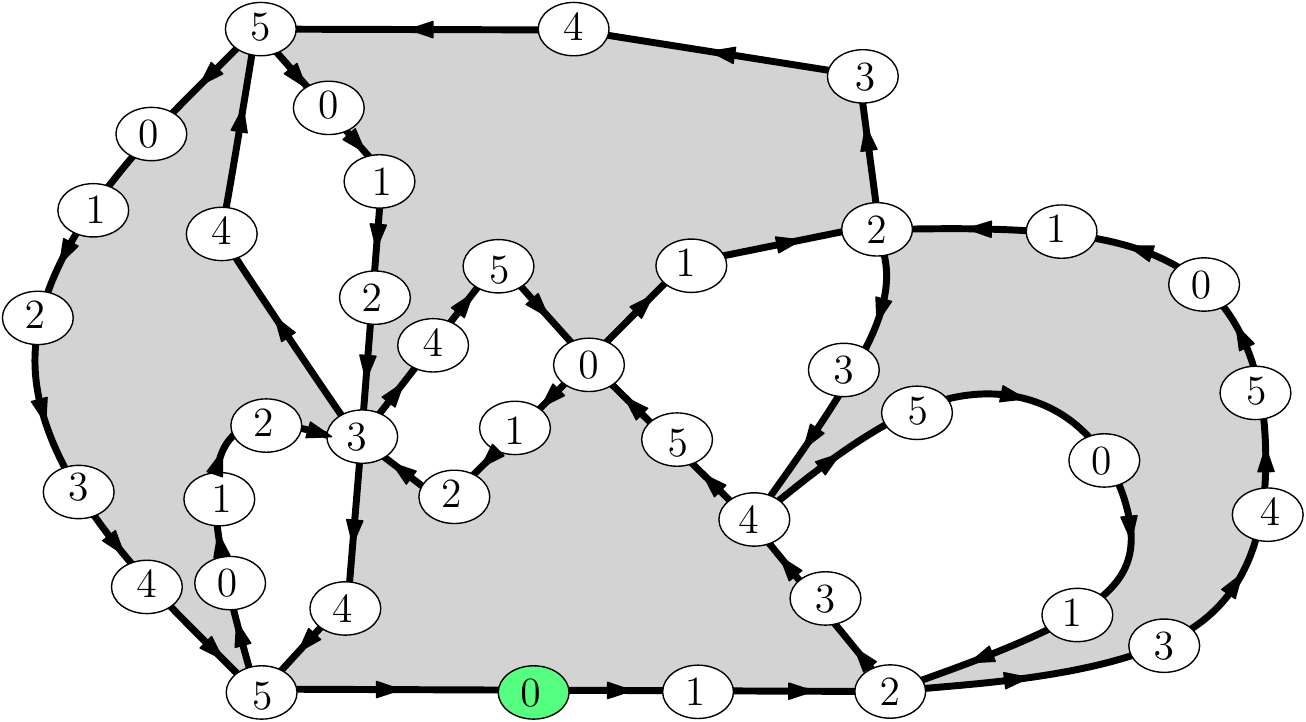}
\hspace{1cm}
\includegraphics[scale=.5]{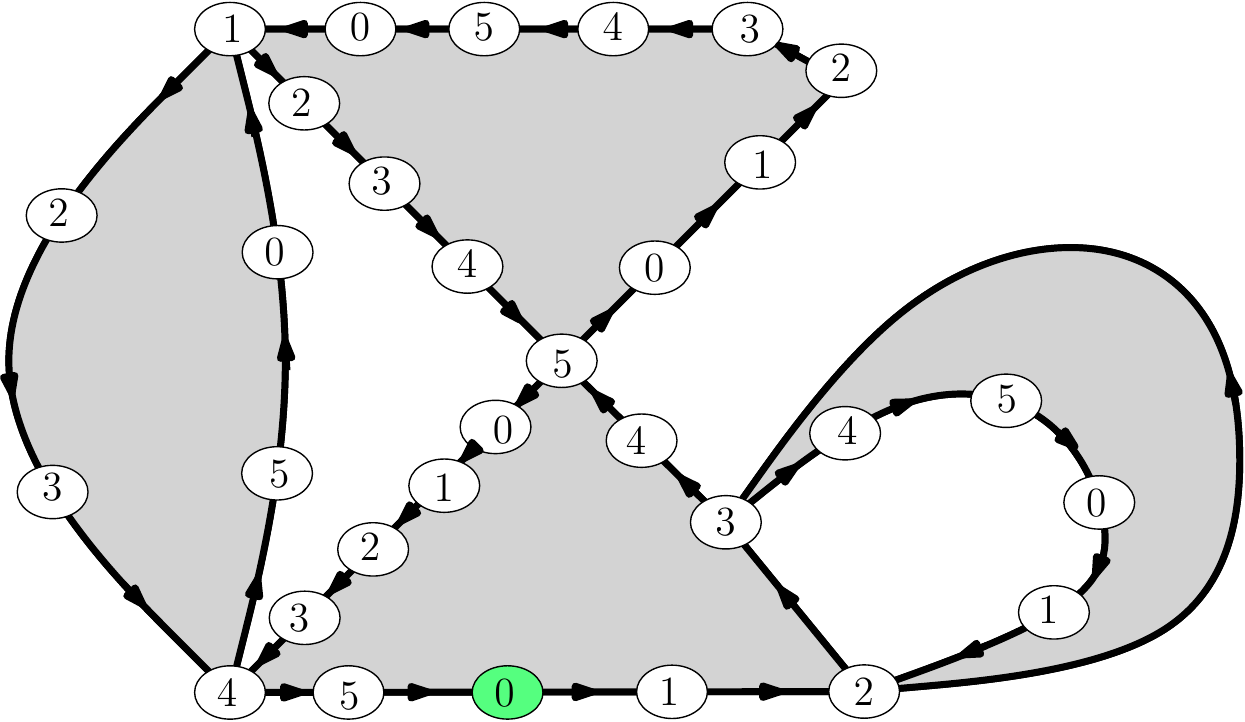}
\end{center}
\caption{Two marked galaxies with respective type $(31^5,2^4)$ and $(321,2^21^2)$.
The right one is a Hurwitz galaxy. 
}\label{fig:bare-galaxy}
\end{figure}

Our starting point is a standard representation of branched coverings
of the sphere in terms of embedded graphs ({aka} ribbon graphs) which,
following the combinatorial literature, we call \emph{maps}.

A \emph{map of genus $g$} is a cellular embedding of a graph in
$\S_g$, the compact oriented surface of genus $g$, that partitions it
into 0-cells (vertices), 1-cells (edges) and 2-cells (faces), and
considered up to orientation preserving homeomorphisms of $\S_g$. A
map is \emph{bicolored} if its faces are colored in black or white so
that adjacent faces have different colors.  The \emph{canonical
  orientation} of a bicolored map is the orientation of edges such
that each edge has a black face on its left. 
For any $r\geq2$, a \emph{$r$-galaxy of genus $g$} is a map of genus
$g$ which is bicolored and such that the length of any oriented cycle
is a multiple of $r$. A galaxy is \emph{marked} if it has a
distinguished vertex.  There is a unique way to color the vertices of
a marked galaxy with colors $\{0,\ldots,r-1\}$ so that the marked
vertex $x_0$ has color 0 and each edge points from a vertex of color
$i$ to a vertex of color $(i+1\mod r)$: the color of any vertex $x$ is
given by the reminder modulo $r$ of the length of any oriented path
from $x_0$ to $x$. By construction the number of edges whose origin
has color $i$ is the same for all $i$, and this common number is the
\emph{size} of the galaxy.  Two examples of marked galaxies are given
in Figure~\ref{fig:bare-galaxy}.

From now on in this section, let $r\geq2$, $g\geq0$ and $d\geq1$ be
fixed integers, and let a \emph{Hurwitz galaxy} be a marked
$(r+1)$-galaxy of size $d$ and genus $g$ whose vertices all have
half-degree one apart from $r$ vertices with half-degree two, one for
each color $c=1,\ldots,r$ (because of the bicoloration of faces,
vertices of galaxies have even degrees and their half-degree represent
their number of incidences with black faces). Since a marked
$(r+1)$-galaxy of genus $g$ and size $d$ has $d$ edges originating
from vertices of color $i$ for each $c=0,\ldots,r$, it is a Hurwitz
galaxy if and only if it has exactly $d-1$ vertices of each color $c$,
$c=1,\ldots,r$, and $d$ vertices of color $0$ (including the marked
vertex).
The \emph{type} of a galaxy $G$ is the pair of partition
$\mu=1^{m_1}\ldots d^{m_d}$, $\nu=1^{n_1}\ldots d^{n_d}$ of $d$ where
$m_i$ is the number of white faces of degree $(r+1)i$ in $G$ and $n_i$
is the number of black faces of degree $(r+1)i$ in $G$.  For any pair
of partitions $\mu$ and $\nu$ with $|\mu|=|\nu|=d$ and $\ell(\mu)=m$,
$\ell(\nu)=n$ and $r=m+\ell-2+2g$, let $\HG_g(\mu,\nu)$ denote the set
of Hurwitz galaxies of type $(\mu,\nu)$.

\begin{proposition}[Folklore, see \emph{e.g.} 
{\cite[Chapter 1]{lando-zvonkin}} or
\cite{Johnson-ribbon}]\label{pro:folklore} 
Let $h^{\bullet}_g(\mu,\nu)$ denote the number of Hurwitz galaxies of type
$(\mu,\nu)$ and genus $g$. Then:
\begin{itemize}
\item The standard Hurwitz numbers,
  $h_g(\mu,\nu)=\frac1dh^\bullet_g(\mu,\nu)$, count equivalence classes
  of branched coverings $f$ of the sphere by $\mathcal{S}_g$ with
  simple ramifications over $r$ fixed points $a_1,\ldots,a_r$ and
  ramification of branching type $\mu$ and $\nu$ over two fixed points
  $w$ and $b$, with a weight $\frac{1}{\aut(f)}$ (two branched
  coverings $f$ and $f'$ are equivalent if there is an homeomorphism
  $h:\S_g\to\S_g$ that maps one onto the other, that is, $f'=f\circ
  h$).
\item The labelled Hurwitz numbers,
  $h^\ell_g(\mu,\nu)=(d-1)!h^\bullet_g(\mu,\nu)$, count transitive
  factorizations $(\tau_1,\ldots,\tau_r,\sigma,\rho)$ of the identity
  permutation of $\mathfrak{S}_d$, where $r=m+n-2+2g$, the
  permutations $\sigma$ and $\rho$ have cyclic type $\mu$ and $\nu$
  respectively, and the $\tau_i$ are transpositions.
\end{itemize}
\end{proposition}
Some authors prefer to work with $h_g(\mu,\nu)$, others with
$h_g^\ell(\mu,\nu)$. Finally, following
\cite{zvonkine2005enumeration}, let the \emph{normalized Hurwitz
  numbers} be defined as follows: let $\bx=(x_1,\ldots,x_m)$
and $\by=(y_1,\ldots,y_n)$ be compositions of respective type
$\mu$ and $\nu$ (that is, $\mathbf{x}$ is a permutation of the parts
of $\mu$),  and 
\[
\bar h_g(\bx,\by)=\frac{\aut(\mu)\aut(\nu)}{(m+n-2+2g)!}h^\bullet_g(\mu,\nu),
\]
where $\aut(\mu)=m_1!\cdots m_d!$ if $m=1^{m_1}2^{m_2}\ldots d^{m_d}$. Then
$(m+n-2+2g)!\bar h_g(\bx,\by)$ is the number of branched coverings as above
with the preimages of $w$ labeled by the integers $1,\ldots,m$ and
the preimages of $b$ labeled by the integers $1,\ldots,n$, in such a way
that the $i$th preimage of $w$ has order $x_i$ and the $i$th preimage
of $b$ has order $y_i$.
The reason for all these trivial variants is that explicit formulas are
best stated with $h_g^\bullet(\mu,\nu)$ or $h_g(\mu,\nu)$, while the piecewise polynomiality
properties hold for $\bar h_g(\bx,\by)$.

Detailed
definitions of branched coverings, maps on surfaces and related
concepts (monodromy, factorizations of permutations) can be found in
\cite[Chapter 1]{lando-zvonkin}.  As explained there, galaxies are
obtained by taking the preimage of a well chosen curve on the sphere,
and more generally, this process yields bijections between branched
covers and various families of maps: these differents bijections are
in a sense \emph{trivial}, they just amount to different
representations of the same underlying branched covering.  Galaxies
explicitely appear in \cite{Johnson-ribbon} where they are refered to
under the generic term \emph{ribbon graphs}. We use the term
\emph{galaxy} because these maps generalize the \emph{constellations}
of \cite{lando-zvonkin}.

\subsection{Distances in galaxies}

\begin{figure}[t]
\begin{center}
\includegraphics[scale=.5]{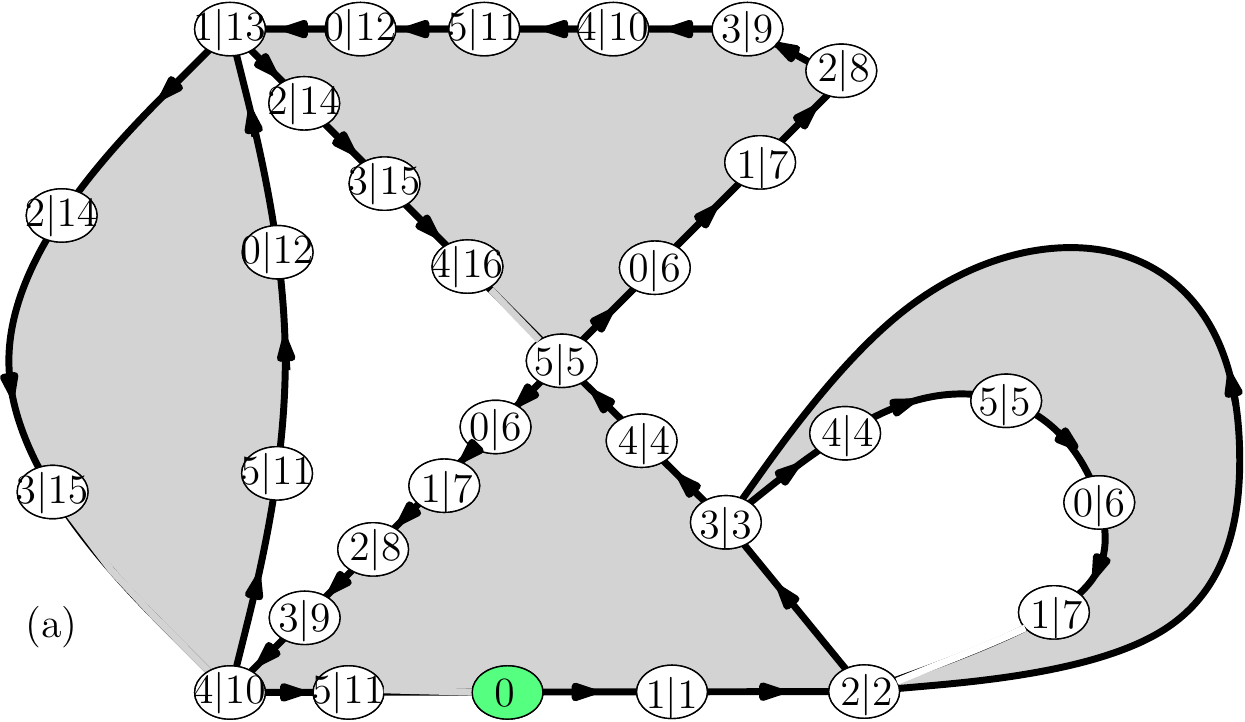}
\qquad\qquad
\includegraphics[scale=.4]{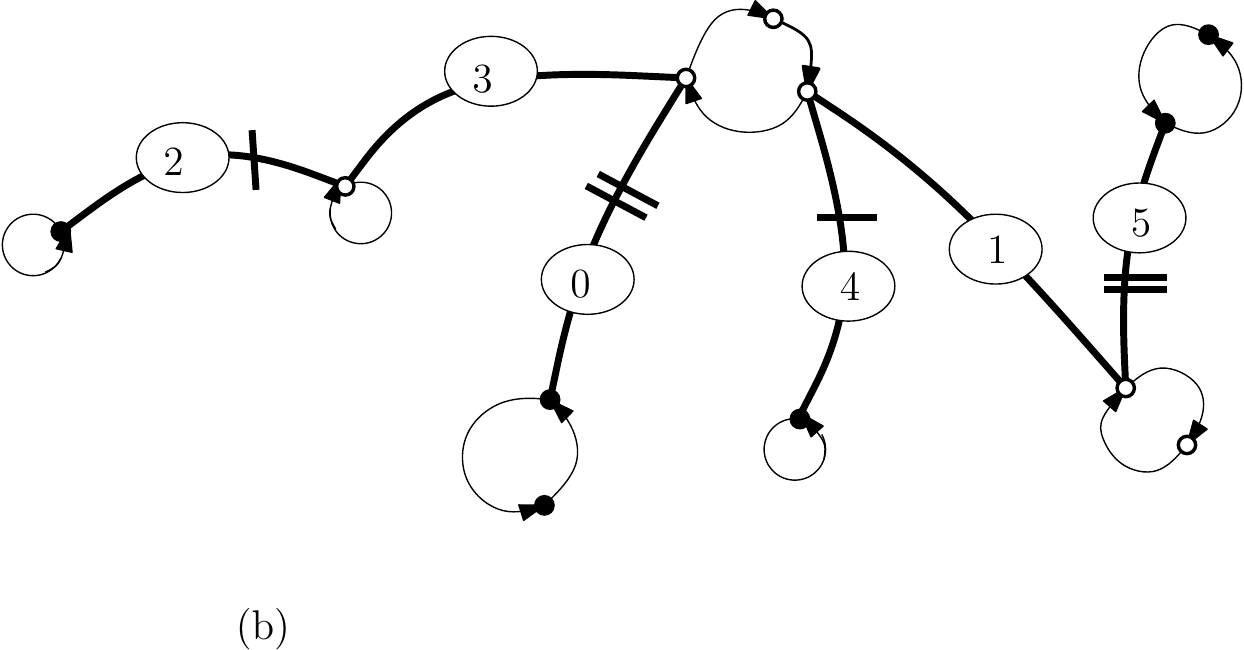}
\end{center}
\caption{(a) Distances and non-geodesic edges: each vertex is labeled
  (color $|$ distance); non-geodesic edges are highlighted with
  lighter stroke.  (b) An edge-labeled free Hurwitz mobile of type
  $(3\,2\,1,2^21^2)$, \emph{i.e.} with one white 3-gon, and one white
  2-gon, and one white 1-gon, and two black 2-gons, and two black 1-gons.
  (The weight of each edge is indicated by the number of transversal
  bars.)  }\label{fig:dist}
\end{figure}

Consider a Hurwitz galaxy endowed with its canonical orientation, and
let $x_0$ denote the marked vertex. The underlying non-oriented map is
connected by definition, and each oriented edge belongs to a cycle
(\emph{e.g.} turning around its incident black face), therefore any
vertex can be reached by an oriented path from $x_0$.  The
\emph{distance labeling} of a vertex $x$ is the number $\delta(x)$ of
edges in a shortest oriented path from $x_0$ to $x$ (see
Figure~\ref{fig:dist}(a)). This distance labeling on a marked galaxy
satisfies several immediate properties:
\begin{enumerate}
\item The color and distance label of a vertex $x$ are related by
  $c(x)=(\delta(x)\mod r+1)$.
\item For any (canonically oriented) edge $e=x\to y$,
  $\delta(y)\equiv \delta(x)+1\mod r+1$, and $\delta(y)\leq \delta(x)+1$.
\end{enumerate}
We can thus define the \emph{weight} of an edge $e=x\to y$ as the
non-negative integer quantity $w(e)=(\delta(x)+1-\delta(y))/(r+1)$.
An edge $e=x\to y$ with weight 0 is called \emph{geodesic}. In other
terms any edge $e=x\to y$ satisfies $\delta(y)=\delta(x)+1-(r+1)w(e)$,
and it is geodesic iff $w(e)=0$. Since the sum of the variations of
labels around each face must be zero, we have the following property:
\begin{enumerate}\setcounter{enumi}{2}
\item The sum of the weight of the edges incident to any face with
  degree $(r+1)i$ is $i$.
\end{enumerate}

\subsection{Free Hurwitz mobiles}

A \emph{Hurwitz mobile} of type $(\mu, \nu)$ and excess $2g$ is a
connected partially oriented graph made of
\begin{itemize}
\item  $d$ white nodes forming $m_i$ disjoint
  oriented simple cycles of length $i$, for $i=1,\ldots,d$; each such
  cycle we refer to as a \emph{white polygon}, 
\item  $d$ black nodes
  forming $n_i$ disjoint oriented simple cycles of length $i$, for
  $i=1,\ldots,d$; each such cycle we refer to as a \emph{black polygon},
\item  $r+1=m+n-1+2g$ non-oriented edges with non-negative weights
 such that
\begin{itemize}
\item each zero weight edge has both endpoints on white polygons 
\item each positive weight edge is incident to a black and a white polygon
\item the sum of the weights of the edges incident to each $i$-gon is
  $i$.
\end{itemize}
\end{itemize}
A Hurwitz mobile is \emph{edge-labeled} if its $m+n-1+2g$ weighted
edges have distinct labels taken in the set $\{0,\ldots,m+n-2+2g\}$.
Let us denote by $\HM_g(\mu,\nu)$ the set of edge-labeled Hurwitz
mobiles of type $(\mu,\nu)$ and excess $2g$. Hurwitz mobiles with excess 0
are called \emph{free Hurwitz mobiles}. An example is given in
Figure~\ref{fig:dist}(b). Since a free Hurwitz mobile is a connected
graph with $m+n-1$ weighted edges connecting $m+n$ polygons, these
edges and polygons form a tree-like structure.

\begin{figure}[t]
\centerline{\includegraphics[scale=.5]{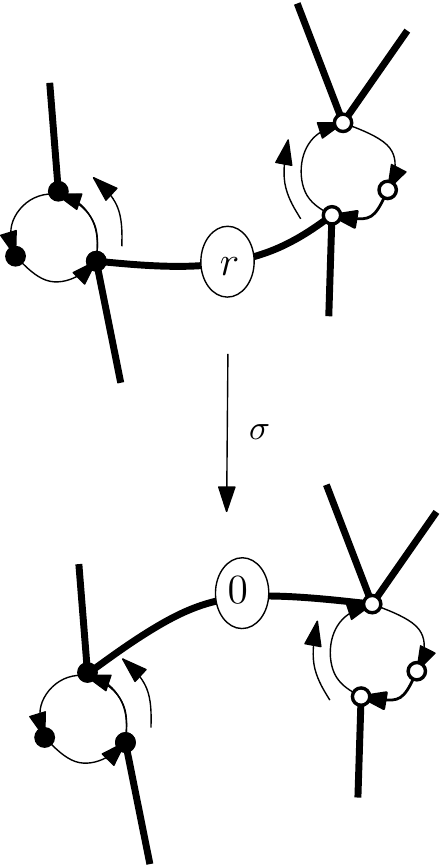}\qquad
\includegraphics[scale=.5]{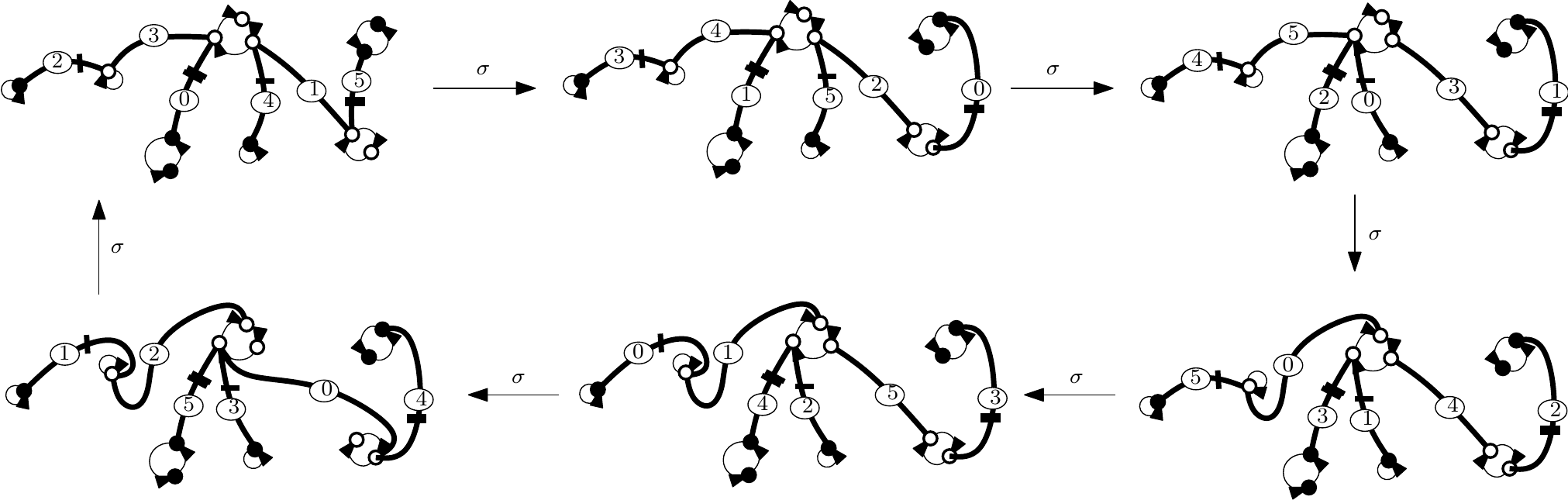}}
\caption{The shift operation. The shift equivalence class of the mobile of Figure~\ref{fig:dist}.}\label{fig:shift}
\end{figure}
Given an edge-labeled Hurwitz mobile $M$, its \emph{shift}
$\sigma(M)$ is the Hurwitz mobile obtained by translating the two
endpoints of the edge with label $r$ along the polygon arc they are
respectively incident to, and then incrementing all edge labels,
modulo $r+1$ (so that the edge with label $r$ gets label $0$).  Two
edge-labeled Hurwitz mobiles are \emph{shift-equivalent} if one can be
obtained from the other by a sequence of shifts. As illustrated by
Figure~\ref{fig:shift}, $\sigma^{r+1}(M)=M$, and we shall later prove
more precisely (Proposition~\ref{pro:shift}) that each
shift-equivalence class contains exactly $r+1$ distinct Hurwitz
mobiles, so that the number of equivalence classes of edge-labeled
Hurwitz mobiles of type $(\mu,\nu)$ and exces $2g$ is
$\frac1{r+1}|\HM_g(\mu,\nu)|$.

Finally, a Hurwitz mobile is \emph{face-labeled} if its white polygons
have distinct labels taken in $\{1,\ldots,m\}$ and its black polygons
have distinct labels taken in $\{1,\ldots,n\}$. The \emph{type} of a
face-labeled Hurwitz mobile is the pair $(\mathbf{x},\mathbf{y})$ of
compositions $\mathbf{x}=(x_1,\ldots,x_m)$,
$\mathbf{y}=(y_1,\ldots,y_n)$ such that the $i$th white polygon is a
$x_i$-gon and the $i$th black polygon is a $y_i$-gon.  Let us denote
$\bar{{\HM}}_g(\mathbf{x},\mathbf{y})$ the set of face-labeled Hurwitz
mobiles with type $(\mathbf{x},\mathbf{y})$ and excess $2g$. Then by
an immediate double counting argument, edge-labeled and face-labeled
Hurwitz mobile numbers are simply related:
\[
{\aut(\mu)}\,{\aut(\nu)}|{\HM}_g(\mu,\nu)|=
(m+n-1+2g)!|\bar{{\HM}}_g(\mathbf{x},\mathbf{y})|.
\]

\subsection{The main bijection $\Phi$}

Given a Hurwitz galaxy $G$ endowed with its distance labeling, we now
construct a (partially oriented) graph $\Phi(G)$ made of oriented
polygons connected by non-oriented edges according to the following
local rules:
\begin{itemize}
\item \emph{(i) Polygons.} Place in each face of degree $i(r+1)$ of
  $G$ an oriented $i$-gon (clockwise in white faces, ccw in black
  ones): the nodes and arcs of these polygons will be the nodes and
  arcs of $\Phi(G)$. 
\end{itemize}
To make easier the description of the next step in the construction,
let us moreover join with dashed lines the $i$ nodes in each face of
degree $i(r+1)$ to the middles of the $i$ edges with color $r\to 0$ on
its boundary:
this divides each face $F$ of degree $(r+1)i$ of $G$ in $i+1$
sub-regions: the interior of the polygon, and $i$ sub-regions each
containing on its boundary a subpath with color $0\to1\to\ldots\to r$
of the boundary of $F$.
\begin{itemize}
\item \emph{(ii) Positive weight edges.} For each non-geodesic edge
  $e=u\to v$ from a vertex $u$ with distance label $\delta(u)=i$ to a
  vertex $v$ with distance label $\delta(v)=i+1-\omega\cdot(r+1)$
  ($\omega\geq1$), let $F_\circ$ and $F_\bullet$ be the white and black faces
  incident to $e$, and let $x$ (resp. $y$) denote the origin of the
  unique arc of $\Phi(G)$ incident to the same sub-region of $F_\circ$
  (resp. $F_\bullet$) as $v$. As illustrated in Figure~\ref{fig:HM}(a),
  create in $\Phi(G)$ an edge with label $c(v)$ and weight $\omega$ between
  $x$ and $y$.
\item \emph{(iii) Zero weight edges.} For each vertex $v$ of $G$ with
  distance label $\delta(v)=j$ and color $c(v)$ that has two incoming
  geodesic edges, let $F_\circ$ and $F_\circ'$ denote the two incident white
  faces, and let $y$ (resp. $y'$) denote the origin of the unique arc
  of $\Phi(G)$ incident to the same sub-region of $F_\circ$ (resp. $F_\circ'$) as
  $v$. As illustrated in Figure~\ref{fig:HM}(b), create in $\Phi(G)$ an
  edge with label $c(v)$ and weight zero between $y$ and $y'$. (For
  later purpose one should imagine that $v$ is split in two by the 
  drawing of this new edge, as suggested by the figure.)
\end{itemize}

\noindent
The application of the main bijection to the exemple of Figure~\ref{fig:dist}(a)
is given in Figure~\ref{fig:HM}(c).

\begin{figure}[t]
\begin{center}
\begin{minipage}{.4\linewidth}
\centerline{\includegraphics[scale=.6,page=2]{Figures/Rules}}
\centerline{\includegraphics[scale=.6,page=1]{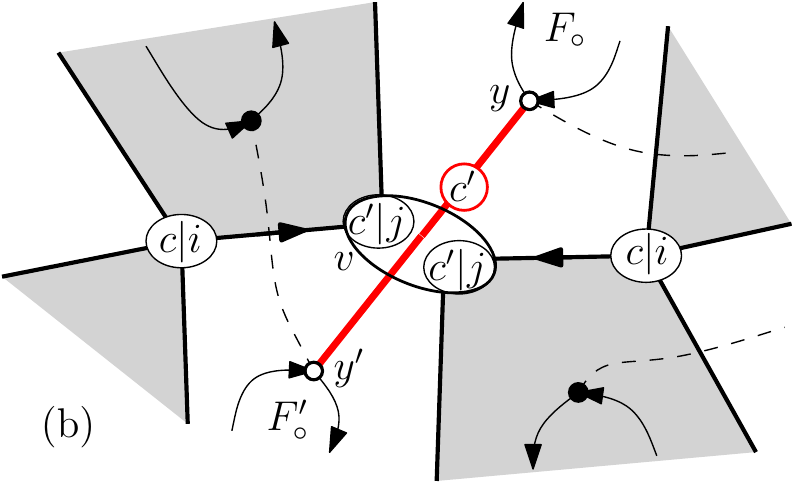}}
\end{minipage}
\begin{minipage}{.5\linewidth}
\includegraphics[scale=.5]{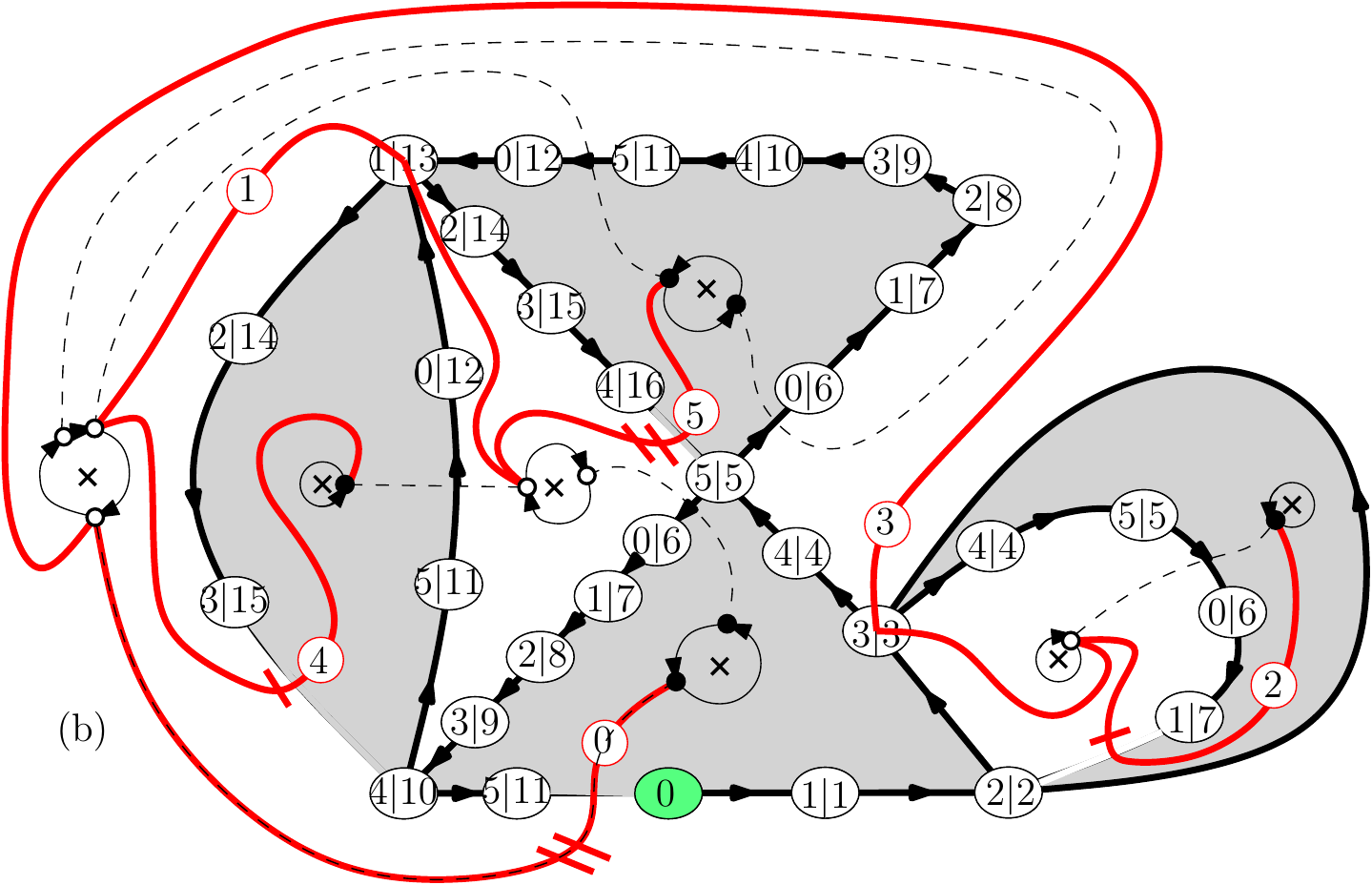}
\end{minipage}
\end{center}
\vspace{-1em}
\caption{Construction rules: (a) for a non-geodesic edge (for
  conciseness $c=c(u)$, $c'=c(v)$), and (b) for a vertex with 2
  incoming geodesic edges ($c'=c(v)$, $c\equiv c'-1$). (c) The construction
  applied to the galaxy $G$ of Fig.~\ref{fig:dist}(a). The resulting
  graph is the Hurwitz mobile of
  Fig.~\ref{fig:dist}(b)}\label{fig:HM}
\end{figure}

\begin{theorem}\label{thm:main-planar}
The image $\Phi(G)$ of a Hurwitz galaxy $G$ with genus $g$ is an
edge-labeled Hurwitz mobile with genus $g$, and the application $\Phi$
is injective. Moreover, in the case $g=0$, it is a 1-to-1
correspondence between Hurwitz galaxies of genus $0$ and
shift-equivalence classes of edge-labeled free Hurwitz mobiles with
the same type. 
\end{theorem}

\begin{corollary}
Genus zero Hurwitz numbers count shift-equivalence classes of
edge-labeled free Hurwitz mobiles,
\[
h^\bullet_0(\mu,\nu)=\frac1{m+n-1}\;|{\HM}_0(\mu,\nu)|,
\]
and, consequently, normalized genus zero Hurwitz numbers count
face-labeled free Hurwitz mobiles,
\[
\bar h_0(\mathbf{x,y})=|\bar{{\HM}}_0(\mathbf{x},\mathbf{y})|.
\]
\end{corollary}

The application $\Phi$ is in fact a 1-to-1 correspondence between
Hurwitz galaxies of genus $g$ and shift-equivalence classes of some
particular Hurwitz mobiles of excess $2g$ called \emph{coherent
  Hurwitz mobiles of genus $g$}. However we postpone the statement of
the corresponding Theorem~\ref{thm:main} to Section~\ref{sec:proof}
where we prove that $\Phi$ is bijective, because the definition of
these higher genus coherent Hurwitz mobiles has a non-trivial twist
which makes enumerative consequences harder to derive.

We would like to insist on the fact that this result is \emph{not}
just a reformulation of Cavalieri \emph{et al} combinatorial
interpretation of the cut-and-join equation \cite{cavalieri10}. Free
Hurwitz mobiles are significantly simpler to count than branched
coverings or the associated tropicalized diagram, even in the planar
case.  To support this assertion, we observe that the representation
of Cavalieri \emph{et al} does not allow, as far as we know, to derive
directly the original Hurwitz formula for $h_0(\mu,1^d)$. Instead, as shown in
Section~\ref{sec:countcacti}, Theorem~\ref{thm:main-planar} easily
results in a bijective proof of this formula.

\section{Enumerative consequences}

\subsection{A bijective proof of Hurwitz' formula}\label{sec:countcacti}
In the case $\nu=1^d$ of simple Hurwitz numbers,  free
Hurwitz mobiles are easy to count:  
\begin{proposition}[Hurwitz formula] The number of Hurwitz mobiles in 
$\HM_0(\mu,1^d)$ is 
\[
|\HM_0(\mu,1^d)|={d+m-1\choose m-1}\cdot \frac1m{m\choose m_1,\ldots,m_d}d^{m-2}\cdot d!\prod_{i\geq1}\left(\frac{i^i}{i!}\right)^{m_i}
\]
and as a consequence
\[
h^\bullet_0(\mu,1^d)=d^{m-2}\cdot(d+m-2)!\cdot\prod_{i\geq1}\frac1{m_i!}\left(\frac{i^i}{i!}\right)^{m_i}.
\]
\end{proposition}
\begin{proof}
By definition, when $\nu=1^d$, all black polygons are 1-gons, each
incident to only one positive edge. As a consequence all positive
edges have weight 1, and these positive edges are pending edges
attached to white polygons. By definition again, each white $i$-gons
is incident to $i$ such pending edges. Finally the white polygons and
zero weight edges form a \emph{Cayley cactus}, that is a tree-like
structure consisting of $m$ polygons connected by $(m-1)$ labeled
edges. Let us conversely consider the number of ways to construct
Hurwitz mobiles by first building a Cayley cactus with edge labels
forming a $(m-1)$-element subset $I$ of $\{0,\ldots,d+m-1\}$, and then
adding zero weight edges and black 1-gons. 

Let $I$ be one of the ${d+m-1\choose m-1}$ subsets of $m-1$ elements
of $\{0,\ldots,d+m-1\}$. The number of Cayley cacti with $m_i$ white
$i$-gons ($i=1,\ldots,d$) and $m-1$ labeled edges having distinct
labels in $I$ is well known to be
\[
\frac1m{m\choose m_1,\ldots,m_d}d^{m-2}.
\]
(This is a simple extension of Cayley formula, which corresponds to
the case $\mu=1^d$. A proof follows from Lagrange inversion formula
applied to the exponential generating function of these cacti or by a
direct Pr\"ufer encoding, see \emph{e.g.}
\cite[Chap. 5]{book:Stanley}). The number of ways to distribute
the remaining $d$ labels to the polygons of a Cayley cactus so that
each $i$-gon gets a subset of $i$ labels is
\[
{d\choose \mu}=\frac{d!}{\prod_{i\geq1}{(i!)^{m_i}}}
\]
Each free Hurwitz mobile is then uniquely obtained from such a cactus
by assigning each of the $i$ extra labels of each $i$-gon to an edge
carrying a black 1-gon attached to one of the $i$ nodes of the
$i$-gon: the total numbers of way to do this assignment is
$\prod_{i\geq1}(i^i)^{m_i}$. The number of free Hurwitz mobiles of
type $(\mu,1^d)$ is thus
\[
{d+m-1\choose m-1}\cdot \frac1m{m\choose m_1,\ldots,m_d}d^{m-2}\cdot d!\prod_{i\geq1}\left(\frac{i^i}{i!}\right)^{m_i}
\]
and Hurwitz formula follows.  \hfill$\Box$
\end{proof}

\subsection{A shape formula for double Hurwitz numbers}\label{sec:applications}
From now on in this section, let $\mathbf{x}=(x_1,\ldots,x_m)$ and
$\mathbf{y}=(y_1,\ldots,y_n)$ be two compositions of $d$ and let
$r=m+n-2$. In order to obtain formulas for $\bar
h_0(\mathbf{x},\mathbf{y})$ we classify face-labeled Hurwitz mobiles
according to their \emph{weighted skeleton}, that is, the bipartite
graph with weighted edges obtained upon contracting each polygon into
a vertex and removing zero weight edges, and even more coarsly
according to their \emph{bare skeleton}, the unweighted bipartited
graph obtained from the weighted skeleton by forgetting the edge
weights.  Let us define $\S_{m,n}$ as the set of \emph{bare shapes},
that is bipartite graphs on the vertex set
$\{1,\ldots,m\}\cup\{1,\ldots,n\}$ without cycles or isolated
vertices.  The bare skeleton of a face-labeled Hurwitz mobile of $\bar
\HM_0(\mathbf{x},\mathbf{y})$ is clearly an element of $\S_{m,n}$, but
to describe the elements of $\S_{m,n}$ that can be bare skeletons of
Hurwitz mobiles in $\bar \HM_0(\textbf{x},\textbf{y})$ we need some
notations.

\begin{figure}[t]
\begin{center}
\includegraphics[scale=.5]{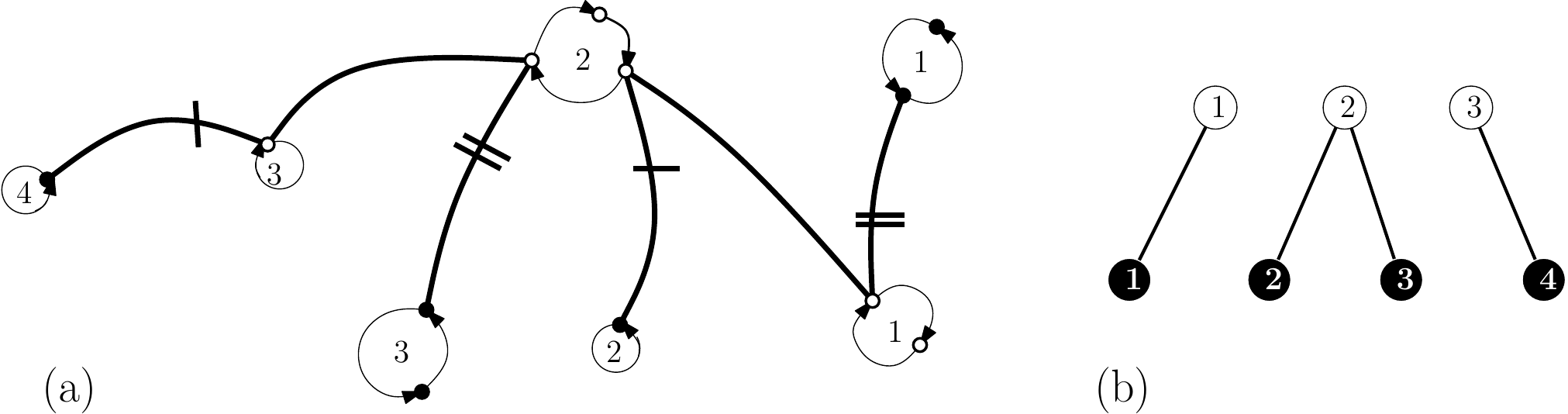}
\end{center}
\caption{(a) A face-labeled Hurwitz mobile with $\mathbf{x}=(2,3,1)$
  and $\mathbf{y}=(2,1,2,1)$. (b) Its bare skeleton. The edge weight vector is
  $\ell=(2,1,2,1)$ with edges enumerated from left to right. The degree
  and excess vectors are $d^\circ=(1,2,1)$,
  $\varepsilon^\circ=(1,1,0)$, $d^\bullet=(1,1,1,1)$ and
  $\varepsilon^\bullet=(1,0,1,0)$.}
\end{figure}

Given a bare shape $s$, let $|s|$ denote its number of edges, and
$c(s)=m+n-|s|$ its number of connected components, and $d^\circ_i(s)$
(resp. $d^\bullet_i(s)$) the degree of the $i$th white (resp. black)
vertex of $s$. In view of the vertex labels, bare shapes have no
nontrivial automorphisms, thus we can assume that their edges are
canonically labeled from 1 to $|s|$. Let us denote by $s^\circ_j$ and
$s^\bullet_j$ the white and black rooted subtrees around the $j$th
edge, and given a subgraph $s'$ of a bare shape $s$, let $W(s')$ and
$B(s')$ respectively denote the sets of indices of white and black
vertices that belong to $s'$. With these notations, for any $s\in
\S_{m,n}$, let $C(s)$ be the region of $\mathbb{R}^{m+n}$ defined by
the inequalities
\[
\sum_{i\in W(s^\circ_j)}x_i>\sum_{i\in B(s^\circ_j)}y_i
\]
for each edge $j$ of $s$, $j=1,\ldots,|s|$, and the equalities
\[
\sum_{i\in W(s_j)}x_i=\sum_{i\in B(s_j)}y_i
\]
for each connected component $s_j$ of $s$, $j=1,\ldots, c(s)$.

\begin{theorem}\label{thm:piecewise}
The normalized Hurwitz number of type $(\mathbf{x},\mathbf{y})$ with $|\mathbf{x}|=|\mathbf{y}|=d$ is 
\[
\bar h_0(\mathbf{x},\mathbf{y})=\sum_{s\in \S_{m,n}}
\left(d^{c(s)-2}\prod_{i=1}^mx_i^{d^\circ_i(s)-1}y_i^{d^\bullet_i(s)-1}
\prod_{j=1}^{c(s)}\left(\sum_{i\in W(s_j)}x_i\right)\cdot
\chi_{(\mathbf{x},\mathbf{y})\in C(s)}\right)
\]
where $\chi$ is the characteristic function: $\chi_{P}=1$ if $P$ is true, $0$ otherwise.

In particular for fixed $m$ and $n$, the number of regions $C(s)$ to
be considered is finite and $\bar h_0(\mathbf{x},\mathbf{y})$ is a
piecewise polynomial.

\end{theorem}
In order to prove the theorem, let us define a \emph{weighted shape}
as a pair $(s,\ell)$ where $s$ is a bare shape and
$\ell=(\ell_1,\ldots,\ell_{|s|})$ is a $|s|$-uple of positive
integers, where $\ell_j$ is to be interpreted as the weight of the
$j$th edge of $s$.  Given a weighted shape $(s,\ell)$, let us denote
by $x_i(s,\ell)$ (resp. $y_i(s,\ell)$) the sum of the weight of edges
incident to the $i$th vertex of $s$, and by
$\varepsilon^\circ_i(s,\ell)$ (resp. $\varepsilon^\bullet_i(s,\ell)$)
the weight excess at $i$:
$x_i(s,\ell)=d^\circ_i(s)+\varepsilon^\circ_i(s,\ell)$
(resp. $y_i(s,\ell)=d^\bullet_i(s)+\varepsilon^\bullet_i(s,\ell)$). The
type of the weighted shape $(s,\ell)$ is the pair of compositions
$(\mathbf{x}(s,\ell),\mathbf{y}(s,\ell))$ whose parts are the
$x_i(s,\ell)$ and the $y_i(s,\ell)$.

Observe that in the above definition the $x_i=x_i(s,\ell)$ are
actually linear combinations of the $\ell_i$: more precisely, if
$E^\circ_i(s)$ denote the set of the edges incident to the $i$th white
vertex in $s$, then $x_i=\sum_{j\in E^{\circ}_i(s)}\ell_j$, and
similarly for the $i$th black vertex, $y_i=\sum_{j\in
  E^\bullet_i(s)}\ell_j$.  Conversely, for any weighted shape
$(s,\ell)$ of type $(\mathbf{x},\mathbf{y})$, the $\ell_j$ can be
recovered from $s$ and $(\mathbf{x},\mathbf{y})$: let us denote by
$s^\circ_j$ and $s^\bullet_j$ the white and black rooted subtrees
around the $j$th edge: then
$\ell_j=\ell_{s,j}(\mathbf{x},\mathbf{y})=\sum_{i\in
  W(s^\circ_j)}x_{i}-\sum_{i\in B(s^\circ_j)}y_i$.  Similarly we also
have the redundent equations
$\ell_j=\ell_{s,j}(\mathbf{x},\mathbf{y})=\sum_{i\in
  B(s^\bullet_j)}y_{i}-\sum_{i\in W(s^\bullet_j)}x_i$.

This discussion implies that if a Hurwitz mobile $M$ of
$\HM_0(\mathbf{x},\mathbf{y})$ has bare skeleton $s$, then
$(\textbf{x},\textbf{y})\in C(s)$: Indeed the weighted skeleton of $M$
has positive weights on every edge by construction.  Moreover the sum
of the weights of edges inside each component is by definition equals
to the sum of the $x_i$s and to the sum of the $y_i$s in this
component.

The theorem then follows from the converse analysis of the number of
face-labeled free Hurwitz mobiles having a given weighted skeleton
$(s,\ell)$.
\begin{lemma}\label{lem:face-labeled}
The number of free face-labeled Hurwitz mobiles with weighted
skeleton $(s,\ell)$ and type $(\mathbf{x},\mathbf{y})$ with
$|\mathbf{x}|=|\mathbf{y}|=d$, edges and $c(s)=m+n-|s|$ components is
\[
\left\{\begin{array}{ll}
R_{s}(\mathbf{x},\mathbf{y})=
\displaystyle
d^{c(s)-2}\left(\prod_{i=1}^mx_i^{d^\circ_i(s)-1}\right)\left(\prod_{i=1}^ny_i^{d^\bullet_i(s)-1}\right)\prod_{j=1}^{c(s)}\left(\sum_{i\in
  W(s_j)}x_i\right) & \textrm{ if $(\mathbf{x},\mathbf{y})\in C(s)$,}\\
0&\textrm{ otherwise}
\end{array}
\right.
\]
where $s_j$ denote the $j$th component of $s$ and $W(s_j)$ its set of
white vertices.
\end{lemma}
\begin{proof}
Let $s$ be a bare shape with $q=c(s)$ connected components, and
let $(\mathbf{x},\mathbf{y})\in C(s)$. We construct the
corresponding free face-labeled Hurwitz mobiles in three steps:
\begin{enumerate}
\item To obtain a free Hurwitz mobile with skeleton $s$, the $i$th
  white vertex of $s$ must first be replaced by a $x_i$-gone, and each
  incident edge must be attached to one of $x_i$ nodes of this
  polygon. The same apply to black vertices. There are
  \[\left(\prod_{i=1}^mx_i^{d^\circ_i-1}\right)\left(\prod_{i=1}^ny_i^{d^\bullet_i-1}\right)\]
  non-equivalent ways to perform these operations.  
\item Each forest of bipartite cacti obtained at the previous step has
  $q$ components $s_1,\ldots,s_q$. Let $W(s_1)\cup
  B(s_1),\ldots,W(s_q)\cup B(s_q)$ denote the white and black node
  sets of these components.  By construction, the $j$th component has
  $d_j=\sum_{i\in W(s_j)} x_i$ white nodes. In each component, mark
  one of these $d_j$ white nodes: there are
\[
\prod_{j=1}^{c(s)}d_j=\prod_{j=1}^{c(s)}\left(\sum_{i\in W(s_j)}x_i\right)
\]
ways to do that.
\item \label{step:edges} In order to form a free Hurwitz mobile from
  such a forest we need to connect the $q$ connected components by
  $q-1$ edges of weight zero. Upon considering the $j$th connected
  component as a unique marked polygon with $d_j$ white nodes, the
  problem reduces to the standard cactus construction: The number of
  ways to form a cactus by adding $q-1$ edges to a set of $q$ marked
  polygons such that the $j$th polygon has $d_j$ nodes is
  $(\sum_{j=1}^qd_j)^{q-2}$ (according to the extended Cayley formula
  for cacti \cite[Chapter 5]{book:Stanley}).  In our case
  $\sum_{j=1}^qd_j=d$ so that the number of ways to construct a
  Hurwitz cactus from a forest as above is just $d^{q-2}$.
\hfill $\Box$
\end{enumerate}
\end{proof}
Consider for instance the case $m=n=2$: the possible shapes
are given in Figure~\ref{fig:example}, together with their
contribution and their associated region.
Observe that Theorem~\ref{thm:piecewise} is slightly more powerful
than the previous piecewise polynomiality theorems in the literature
(see \cite{cavalieri10}): in particular it immediately implies Hurwitz
formula, which as far as we understand, does not easily follow from
the latter.

\begin{corollary}[Hurwitz's formula]
Let $s$ be a shape contributing to Hurwitz formula, \emph{i.e.} such
that $(\mathbf{x},1^d)\in C(s)$. In view of the condition $y_i=1$ for
all $i$, all black vertices in $s$ are leaves, and $s$ is a collection
of stars, that is, $c(s)=m$, $d^\bullet_i(s)=1$ and
$d^\circ_i(s)=x_i$. The number of such star forests is
$
d!\;\prod_{i=1}^m\frac1{x_i!}
$ 
and each contributes to a factor $d^{m-2}\prod_{i=1}^mx_i^{x_i}$ so that 
\[
\bar h_0(\mathbf{x},1^d)=d!\;d^{m-2}\;\prod_{i=1}^m\frac{x_i^{x_i}}{x_i!}.
\]
\end{corollary}

\begin{corollary}
Let $s$ be a shape contributing to $\bar h_0(\mathbf{x},d)$. In view
of the condition $y_1=d$, $s$ is the unique (star) tree with one black
vertex, that is, $c(s)=1$, $d^\bullet_1(s)=d$, $d^\circ_i(s)=1$, so that:
\[
\bar h_0(\mathbf{x},d)=d^{d-2}
\]
More generally in all the cases were there is only one possible shape,
a product formula holds:
\[
\bar h_0(\mathbf{x},\mathbf{y})=d^{c(s)-2}\left(\prod_{i=1}^mx_i^{d^\circ_i(s)-1}\right)\left(\prod_{i=1}^ny_i^{d^\bullet_i(s)-1}\right)\prod_{j=1}^{c(s)}\left(\sum_{i\in W(s_j)}x_i\right)
\]
\end{corollary}

\subsection{Polynomiality in chambers}

Let $\mathbb{R}^{m+n}_=$ denote the set of pairs of vectors
$(\mathbf{x},\mathbf{y})$ with $\sum_{i=1}^mx_i=\sum_{j=1}^n{y_j}$.
Given two non empty subsets $I\subsetneq\{1,\ldots,m\}$ and
$J\subsetneq\{1,\ldots,n\}$, the subspace of $\mathbb{R}^{m+n}_=$ with
equation
\[
\sum_{i\in I}x_i=\sum_{j\in J}y_{j}
\]
is called a \emph{resonnance hyperplane}. A \emph{$(m,n)$-chamber} is
a connected component of the complement in $\mathbb{R}^{m+n}_=$ of all
the resonnance hyperplanes.

\begin{figure}[t]
\centerline{\includegraphics[scale=.7]{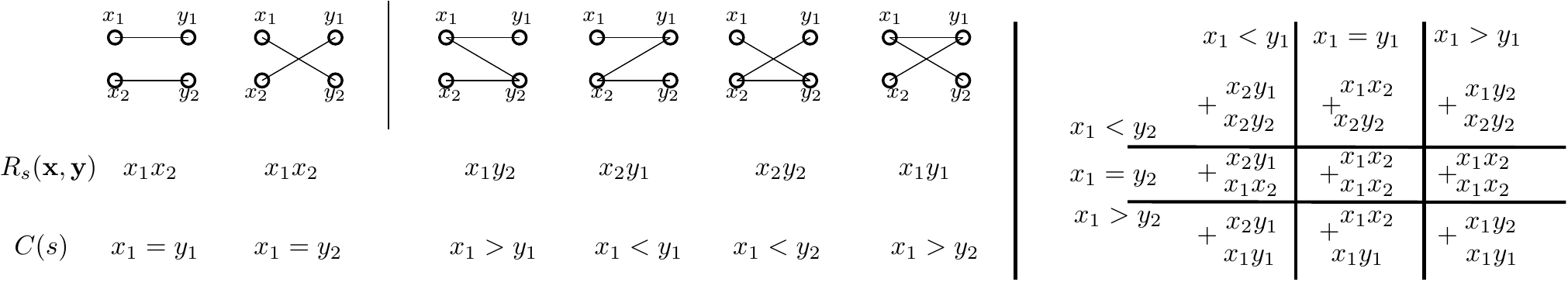}}
\caption{Shapes in $\S_{2,2}$ and their contribution and
  active region, and the resulting chamber diagram, with Hurwitz
  polynomials in chambers and on resonnances, written as expanded sum of
  positive monomials.}
\label{fig:example}
\end{figure}

For any shape $s\in\S_{m,n}$, the polyhedra $C(s)$ has its boundary
included in the union of the resonnance hyperplanes. Therefore any
$(m,n)$-chamber is included either in $C(s)$ or in its
complement. This allows us to define for a $(m,n)$-chamber $\kappa$,
the set $S(\kappa)$ of shapes $s$ in $\S_{m,n}$ such that $\kappa$ is
included in $C(s)$:
\[
S(\kappa)=\{s\in \S_{m,n}\mid \kappa\subset C(s)\}
\]
Observe that the region $C(s)$ associated to a non-connected shape $s$
is always included in a resonnance hyperplane, so that each
$S(\kappa)$ only consists of connected shapes. This implies that the
formula for Hurwitz number slightly simplifies inside chambers:
\begin{corollary}\label{cor:chambers}
Normalized Hurwitz numbers are polynomials inside chambers: for all $(\mathbf{x},\mathbf{y})\in \kappa$,
\[
\bar h_0(\mathbf{x},\mathbf{y})=\sum_{s\in \S(\kappa)}
\prod_{i=1}^mx_i^{d^\circ_i(s)-1}\prod_{i=1}^ny_i^{d^\bullet_i(s)-1}
\]
In particular this sum is a sum of positive monomials.
\end{corollary}

\subsection{Regular points and the Kazarian-Zvonkine 
conjecture}\label{sec:secondappli}

Given a composition ${\mathbf{x}}=(x_1,\ldots,x_m)$ of size
$|\mathbf{x}|\leq d$, let
${\mathbf{x}}^d=(x_1,\ldots,x_m,1^{d-|\mathbf{x}|})$ be its completion
with $d-|\mathbf{x}|$ parts equal to one, which is a composition of
size $d$.  The following theorem makes explicit the dependency of $\bar
h_0(\mathbf{x}^d,\mathbf{y}^d)$ in $d$ when $\mathbf{x}$ and
$\mathbf{y}$ are fixed. Let $\mathcal{B}_{m,n}$ be the set of
bipartite forests on $m$ white and $n$ black vertices (so that
$\mathcal{S}_{m,n}\subset\mathcal{B}_{m,n}$ but elements of
$\mathcal{B}_{m,n}$ can have isolated vertices).

\begin{theorem}\label{thm:withfixpoints} Let 
$\mathbf{x}=(x_1,\ldots,x_m)$, $\mathbf{y}=(y_1,\ldots,y_n)$ be two
  fixed compositions with $|\mathbf{x}|\geq|\mathbf{y}|$. For all
  $d\geq|\mathbf{x}|$, we have
\begin{eqnarray*}
\frac{\bar h_0({\mathbf{x}}^d,{\mathbf{y}}^d)}
{(d-|\mathbf{x}|)!(d-|\mathbf{y}|)!}
 &=&
\prod_{i=1}^m\frac{x_i^{x_i}}{x_i!}
\prod_{i=1}^n\frac{y_i^{y_i}}{y_i!}
\cdot
\frac{d^{d+m+n-2}}{d!}
\cdot
\sum_{k=0}^{m+n-1}\frac1{d^k}
\sum_{e=0}^{|\mathbf{y}|-k}
\frac{(d)_{|\mathbf{x}|+|\mathbf{y}|-k-e}}{d^{|\mathbf{x}|+|\mathbf{y}|-k-e}}
Q_{k,e}({\mathbf{x}},{\mathbf{y}})
\end{eqnarray*}
where the coefficients $Q_{k,e}({\mathbf{x}},{\mathbf{y}})$
are multivariate Laurent polynomials in the $x_1,\ldots,x_m,y_1,\ldots,y_n$:
\begin{eqnarray*}
Q_{k,e}(X,Y)&=&
\sum_{\begin{subarray}{c}s\in
    \mathcal{B}_{m,n}\\ |s|=k\end{subarray}}
\sum_{\begin{subarray}{c}\varepsilon_1+\ldots+\varepsilon_{k}=e\\
\varepsilon_1\geq0,\ldots,\varepsilon_k\geq0\end{subarray}} 
P_{s,\varepsilon}(X,Y)
\end{eqnarray*}
with 
\[
P_{s,\varepsilon}(X,Y)\;=\;
\prod_{i=1}^m
\frac{(X_i)_{d^\circ_i(s)+\epsilon^\circ_i(s,\varepsilon)}}{X_i^{1+\epsilon^\circ_i(s,\varepsilon)}}
\prod_{i=1}^n
\frac{(Y_i)_{d^\bullet_i(s)+\epsilon^\bullet_i(s,\varepsilon)}}{Y_i^{1+\epsilon^\bullet_i(s,\varepsilon)}}
\prod_{j=1}^{c(s)}
\left(\sum_{i\in
  W(s_j)}X_i+\sum_{i\in B(s_j)}Y_i\;-|s_j|-\varepsilon(s_j)\right)
\]
Moreover the $P_{s,\varepsilon}(X,Y)$ are multiplicative on the
components: if $s=s'\cup s''$ then
\[
P_{s,\varepsilon}(X,Y)=P_{s',\varepsilon'}(X',Y')\cdot P_{s'',\varepsilon''}(X'',Y'')
\]
where the partition of $s$ naturally induces the partitions of
variables and excesses. In particular only the
$P_{s,\varepsilon}(X,Y)$ for connected shapes $s$ need to be computed.
\end{theorem}
In particular this theorem refines and immediately implies
Formula~(\ref{for:Kazarian-Zvonkine}).  The  expressions
\begin{eqnarray*}
P_{\circ,.}(x;\emptyset)&=&P_{\bullet,.}(\emptyset;y)\,=\,1
\qquad \textrm{and} \qquad
P_{\circ-\bullet,\varepsilon}(x;y)=
\frac{(x)_{1+\epsilon}}{x^{1+\epsilon}}
\frac{(y)_{1+\epsilon}}{y^{1+\epsilon}}
\left(x+y-1-\varepsilon\right),
\end{eqnarray*}
allow to make the theorem completely explicit in the case $m=n=1$ and
yields a formula for $\bar h_0(\alpha1^{d-\alpha},\beta1^{d-\beta})$
for $\alpha\geq\beta\geq1$, which immediately boils down to
Formula~(\ref{for:explicit}). 

\medskip
\begin{proof}[of Theorem~\ref{thm:withfixpoints}]
In order to prove Theorem~\ref{thm:withfixpoints}, we need to ignore
some leaves in the skeleton of a Hurwitz mobile: Given a weighted
shape $(s,\ell)$ of type $(\mathbf{x}',\mathbf{y}')$ with
$\ell(\mathbf{x}')=m'$ and $\ell(\mathbf{y}')=n'$, and $m\leq m'$ and
$n\leq n'$ such that $x'_i=y'_j=1$ for all $m<i\leq m'$ and $n<j\leq
n'$, let $(s,\ell)|_{[m],[n]}$ denote the degenerated weighted shape
obtained from $(s,\ell)$ by deleting all white vertices with indices
larger than $m$ and all black vertices with indices larger than $n$
and the incident edges (all these vertices are leaves by
hypothesis). More precisely a \emph{degenerated weighted shape} is a
pair $(s,\ell)$ where $s$ is a bipartite forest and $\ell$ is a vector
of edge weights $(\ell_1,\ldots,\ell_{|s|})$.

\begin{lemma}
Let $({\mathbf{x}}^d,{\mathbf{y}}^d)$ as above, with $d\geq
\max(|\mathbf{x}|,|\mathbf{y}|)$ an integer.  Then for any degenerated
weighted shape $(s,\ell)$ of weight $|\ell|$, the number of weighted
shapes $(s',\ell')$ of type $({\mathbf{x}}^d,{\mathbf{y}}^d)$ such that
$(s',\ell')|_{[m],[n]}=(s,\ell)$ is
\[
\left((d-|\mathbf{y}|)!\prod_{i=1}^m\frac{(x_i)_{x_i(s,\ell)}}{x_i!}\right)\cdot\left((d-|\mathbf{x}|)!\prod_{i=1}^n\frac{(y_i)_{y_i(s,\ell)}}{y_i!}\right)\cdot\frac1
  {(d+|\ell|-|\mathbf{x}|-|\mathbf{y}|)!}
\]
where $(x)_k=x(x-1)\ldots(x-k+1)$ denotes the descending factorial.  In
particular this number is zero if $|\ell|>\min(|\mathbf{x}|,|\mathbf{y}|)$,
or if $d-|\ell|>(d-|\mathbf{x}|)+(d-|\mathbf{y}|)$.
\end{lemma}
\begin{proof}
Weighted shapes as in the lemma are obtained by inserting vertices of
degree and weight one in all possible ways on the white and black
vertices and by matching together the remaining black and white
vertices of degree one (if any): the $i$th white vertex of $s$ has weight
$x_i(s,\ell)$ in $(s,\ell)$ and must have weight $x_i$ in $(s',\ell')$
so that it must get $x_i-x_i(s,\ell)$ among the $d-|\mathbf{y}|$ black
vertices of degree one to be added.  Similarly the $i$th black vertex
of $s$ must get $y_i-y_i(s,\ell)$ among the $d-|\mathbf{y}|$ white
vertices of degree one to be added. The number of remaining black (or
white) vertices of degree 1 to be matched in $s'$ is therefore
\[
d-|\mathbf{y}|-\sum_{i=1}^mx_i+\sum_{i=1}^mx_i(s,\ell)=
d-|\mathbf{x}|-\sum_{i=1}^ny_i+\sum_{i=1}^ny_i(s,\ell)=d+|\ell|-|\mathbf{x}|-|\mathbf{y}|
\]
Observe that
$d+|\ell|-|\mathbf{x}|-|\mathbf{y}|=(d-|\mathbf{x}|)+(d-|\mathbf{y}|)-(d-|\ell|)$: the
number of matching edges that have been deleted is the number of edges
that are counted twice when counting the number of deleted vertices.
\hfill$\Box$
\end{proof}
The contribution of a weighted shape $(s',\ell')$ such that
$(s',\ell')_{[m],[n]}=(s,\ell)$ to $\bar
h_0(\mathbf{x}^d,\mathbf{y}^d)$ is then
\begin{eqnarray*}
R_{s',\ell'}({\mathbf{x}}^d,{\mathbf{y}}^d)&=&d^{c(s')-2}\prod_{i=1}^mx_i^{d^\circ_i(s')-1}\prod_{i=1}^ny_i^{d^\bullet_i(s')-1}\prod_{j=1}^{c(s')}\left(\sum_{i\in W(s'_j)}x_i+\sum_{i\in B(s'_j)}y_i-\sum_{e\in E(s'_j)}\ell(e)\right)\\
&=&d^{c(s')-2}\prod_{i=1}^mx_i^{x_i-\varepsilon^\circ_i(s',\ell')-1}\prod_{i=1}^ny_i^{y_i-\varepsilon^\bullet_i(s',\ell')-1}\prod_{j=1}^{c(s')}\left(\sum_{i\in W(s'_j)}x_i+\sum_{i\in B(s'_j)}y_i-\sum_{e\in E(s'_j)}\ell(e)\right)\\
&=&d^{c(s)-2+(d+|\ell|-|\mathbf{x}|-|\mathbf{y}|)}\prod_{i=1}^mx_i^{x_i-\varepsilon^\circ_i(s,\ell)-1}\prod_{i=1}^ny_i^{y_i-\varepsilon^\bullet_i(s,\ell)-1}\prod_{j=1}^{c(s)}\left(\sum_{i\in W(s_j)}x_i+\sum_{i\in B(s_j)}y_i-\sum_{e\in E(s_j)}\ell(e)\right).
\end{eqnarray*}
In particular this quantity depends on $d$ only through the first
factor $d^{c(s)-2+d+|\ell|-|\mathbf{x}|-|\mathbf{y}|}$. Summing over
all degenerated weighted shapes and taking into account the
multiplicities given by the lemma yield
Theorem~\ref{thm:withfixpoints}. \hfill$\Box$
\end{proof}

\subsection{Polynomiality and interpolating between Theorems~\ref{thm:piecewise} and~\ref{thm:withfixpoints}}
Now we consider the case $d=|\mu|\geq|\nu|$ more precisely: for $\nu$
the empty composition and $\mu$ arbitrarily varying we recover Hurwitz
formula, and in general for $\nu$ fixed and $\mu$ arbitrarily varying,
we obtain for the corresponding genus zero almost simple Hurwitz
numbers $\bar h_0(\mathbf{x},\mathbf{y}^d)$ a polynomiality property
akin to that of the higher genus simple Hurwitz numbers
\cite{dunin-barkowski13}.

\begin{corollary}\label{cor:almost-simple} Let 
$\mathbf{y}$ be a fixed composition. Then for any composition
  $\mathbf{x}$ with $d=|\mathbf{x}|\geq|\mathbf{y}|$,
\begin{eqnarray*}
\frac{\bar h_0(\mathbf{x},{\mathbf{y}}^d)} {(d-|\mathbf{y}|)!}  &=&
\left(\prod_{i=1}^m\frac{x_i^{x_i}}{x_i!}\right) 
\sum_{\begin{subarray}{c}(s,\varepsilon)\in
    \S(\mathbf{y})\end{subarray}} 
 {d^{m+n-2-|s|}} 
\prod_{i=1}^n
y_i^{d^\bullet_i(s)-1} 
\sum_{1\leq i_1<\ldots< i_{m(s)}\leq m}\tilde P_{s,\varepsilon}(x_{i_1},\ldots,x_{i_{m(s)}})
\end{eqnarray*}
where $\S(\mathbf{y})$ is the set of weighted strict
bipartite graphs $s$ with $n$ black vertices labeled $1,\ldots,n$,
such that the sum of the weight of edges incident to the $i$th black
vertex is $y_i$, and for $s\in\S(\mathbf{y})$, $|s|$ and
$m(s)$ denote respectively the number of edges and of white vertices
of $s$, and, where the $\tilde P_{s,\varepsilon}(X)$ are multivariate
Laurent polynomials:
\[
\tilde P_{s,\varepsilon}(X_1,\ldots,X_p)\;=\;
\prod_{i=1}^p
\frac{(X_i)_{d^\circ_i(s)+\epsilon^\circ_i(s,\varepsilon)}}{X_i^{1+\epsilon^\circ_i(s,\varepsilon)}}
\prod_{j=1}^{c(s)}
\left(\sum_{i\in
  W(s_j)}X_i\right)
\]
Again, the polynomial $\tilde P_{s,\varepsilon}(X)$ are multiplicative on the
components of $s$.
\end{corollary}
In order to make more explicit the nature of our polynomiality result,
let us introduce some notations: Let $\textbf{m}_{\lambda,\lambda'}$
denote the monomial symmetric Laurent polynomial
\begin{eqnarray*}
\mathbf{m}_{\lambda,\lambda'}(x_1,\ldots,x_m)=\sum\frac{x_{i_1}^{\lambda_1}\cdots x_{i_\ell}^{\lambda_\ell}}{x_{j_1}^{\lambda'_{1}}\cdots x_{j_{\ell'}}^{\lambda'_{\ell'}}}
\end{eqnarray*}
where the sum is over all distinct monomials with shape
$(\lambda,\lambda')$.  By convention, 
$\mathbf{m}_{\varepsilon,\varepsilon}(x_1,\ldots,x_m)=1$,
so that the first few  monomial symmetric Laurent polynomials are:
\begin{eqnarray}
&\displaystyle\mathbf{m}_{\varepsilon;\varepsilon}(x_1,\ldots,x_m)=1,\qquad\qquad\qquad
\mathbf{m}_{1;\varepsilon}(x_1,\ldots,x_m)=x_1+\ldots+x_m,\\
&\displaystyle
\mathbf{m}_{\varepsilon;1}(x_1,\ldots,x_m)=\frac1{x_1}+\ldots+\frac1{x_m},\qquad
\mathbf{m}_{2;\varepsilon}(x_1,\ldots,x_m)=\sum_{1\leq i\leq m}
{x_i^2},\\
&\displaystyle\mathbf{m}_{1^2;\varepsilon}(x_1,\ldots,x_m)=\sum_{1\leq i<j\leq m}
{x_ix_j},\qquad
\mathbf{m}_{1;1}(x_1,\ldots,x_m)=\sum_{1\leq i\neq j\leq m}
\frac{x_i}{x_j},\\
&\displaystyle
\mathbf{m}_{\varepsilon;1^2}(x_1,\ldots,x_m)=\sum_{1\leq i<j\leq m}\frac1{x_ix_j},\qquad
\mathbf{m}_{\varepsilon;2}(x_1,\ldots,x_m)=\sum_{1\leq i\leq m}\frac1{x_i^2}
\end{eqnarray}

The monomial symmetric Laurent polynomials satisfy the following consistency relation:
\begin{lemma}\label{lem:monom} Let $\lambda$ and $\lambda'$ be two partitions with respectively $\ell$ and $\ell'$ parts. Then
\begin{eqnarray*}
\sum_{1\leq i_1<\ldots<i_k\leq m}
\mathbf{m}_{\lambda;\lambda'}(x_{i_1},\ldots,x_{i_k})&=&
{m-\ell-\ell'\choose k-\ell-\ell'}
\mathbf{m}_{\lambda;\lambda'}(x_{1},\ldots,x_{m}).
\end{eqnarray*}
\end{lemma} 
\begin{corollary}\label{cor:resum}
Let $f(x_1,\ldots,x_k)$ be a symmetric Laurent polynomial in the variables $x_1,\ldots,x_k$, whose decomposition in the monomial basis reads
\[
f(x_1,\ldots,x_k)=\sum_{(\lambda;\lambda')\in\Lambda}c_{\lambda;\lambda'}\mathbf{m}_{\lambda;\lambda'}(x_1,\ldots,x_k)
\]
where $\Lambda$ is a finite set of pairs of partitions and the $c_{\lambda;\lambda'}$ are constants.
Then 
\[
\sum_{1\leq i_1<\ldots <i_k\leq m}
f(x_{i_1},\ldots,x_{i_k})=\sum_{(\lambda;\lambda')\in\Lambda}\binom{m-\ell-\ell'}{k-\ell-\ell'}
c_{\lambda;\lambda'}\mathbf{m}_{\lambda;\lambda'}(x_1,\ldots,x_m)
\]
and in particular the coefficients in this expansion are polynomials of degree at most $k$ in $m$.
\end{corollary}
With these notations, we can reformulate
Corollary~\ref{cor:almost-simple} in the following form, which is a
restatement of Formula~(\ref{for:sym}) with standard Hurwitz numbers
replaced with normalized ones.
\begin{corollary}\label{cor:pol}
Let $\mathbf{y}$ be a composition, and $\Lambda(y)$ denote the set of
pairs of partitions $(\lambda;\lambda')$ such that
$|\lambda|+|\lambda'|<|\mathbf{y}|$.  Then there exist polynomials
$q^{\mathbf{y}}_{\lambda;\lambda'}(m)$ in $m$ such that for all
$m\geq1$ and $\mathbf{x}=(x_1,\ldots,x_m)$ with
$|\mathbf{x}|=d\geq|\mathbf{y}|$,
\[
\frac{\bar h_0(\mathbf{x},\mathbf{y}^d)}{(d-|\mathbf{y}|)!}
=\prod_{i=1}^m\frac{x_i^{x_i}}{x_i!}\cdot d^{m-1-|\mathbf{y}|}\cdot
\sum_{(\lambda;\lambda')\in\Lambda(\mathbf{y})}q^{\mathbf{y}}_{\lambda;\lambda'}(m)
\mathbf{m}_{\lambda;\lambda'}(x_1,\ldots,x_m)
\]
\end{corollary}
\begin{proof}
The result follows from Corollary~\ref{cor:almost-simple} upon
observing that shapes that are symmetric in the $x_i$ can be grouped
together to form terms of the form of 
Corollary~\ref{cor:resum}.\hfill $\Box$
\end{proof}

For instance, the polynomial associated to the $k$-star graph
$s^{(k)}$ consisting of one black vertex of degree $k$ is
\begin{eqnarray*}
\tilde P_{(s^{(k)},\varepsilon)}(x_1,\ldots,x_k)&=&
\prod_{i=1}^k\frac{(x_i)_{1+\varepsilon_i}}{x_i^{1+\varepsilon_i}}
\left(x_1+\ldots+x_k\right)
\end{eqnarray*}
and Corollary~\ref{cor:almost-simple} gives the following generalization of Formula~(\ref{for:explicit}): 
\begin{eqnarray*}
\frac{\bar h_0(\mathbf{x},\beta1^{d-\beta})}{(d-\beta)!}
&=&
\prod_{i=1}^m\frac{x_i^{x_i}}{x_i!}  \cdot {d^{m-1}} \cdot
\sum_{k=1}^{\beta}\frac{\beta^{k-1}}{d^k}
\sum_{1\leq i_1<\ldots< i_k\leq m}
\left(x_{i_1}+\ldots+x_{i_k}\right)
\sum_{\begin{subarray}c\ell_1+\ldots+\ell_k=\beta\\\ell_i\geq1\end{subarray}}
\prod_{j=1}^k\frac{(x_{i_j})_{\ell_j}}{x_{i_j}^{\ell_j}}\\
\end{eqnarray*}
In particular the case $\nu=2$ gives again Hurwitz formula,
\begin{eqnarray*}
\frac{\bar h_0(\mathrm{x},21^{d-2})}{(d-2)!}
&=&
\prod_{i=1}^m\frac{x_i^{x_i}}{x_i!}  \cdot {d^{m-2}} \cdot
\left(
 m+
{d-2}
\right)
\end{eqnarray*}
while $\nu=3$ gives
\begin{eqnarray*}
\frac{\bar h_0(\mathbf{x},31^{d-3})}{(d-3)!}
&=&
\prod_{i=1}^m\frac{x_i^{x_i}}{x_i!}  \cdot {d^{m-3}} \cdot
\left( 
\mathbf{m}_{2;\varepsilon}+2\mathbf{m}_{1^2;\varepsilon}-3(m+2)\mathbf{m}_{1;\varepsilon}-\mathbf{m}_{1;1}
+\frac32m^2+\frac12m
\right).
\end{eqnarray*}

\section{The proof that the mapping $\Phi$ is bijective}\label{sec:proof}
Rather than proving Theorem~\ref{thm:main-planar} directly, we give an
alternative construction that proceeds in two steps, each of which is
bijective:
\begin{itemize}
\item The first step consists in cutting the surface $\S_g$ underlying
  $G$ along a tree $\Theta(G)$ to get a cactus $C=\Gamma(G)$: more
  precisely we show (Prop.~\ref{pro:firstbij}, Prop.~\ref{pro:shift}
  and Cor.~\ref{cor:shift}) that there exist sets of cacti
  $\HC^{0c}_g(\mu,\nu)\subset\HC^c(\mu,\nu)_g\subset\HC_g(\mu,\nu)$, a
  shift $\sigma'$ on $\HC^c_g(\mu,\nu)$ and a mapping
  $\Gamma:\HG_g(\mu,\nu)\to\HC_g(\mu,\nu)$ such that
\[
\HG_g(\mu,\nu)\mathop{\longrightarrow}^\equiv_\Gamma
\HC^{0c}_g(\mu,\nu)\equiv\HC^c_g(\mu,\nu)_{/\sigma'}.
\]
\item The second step consists in simplifying the cactus $C$ into a
  Hurwitz mobile $\Pi(C)$: more
  precisely we show ({Prop.~\ref{pro:retract}}) that there exist a set of mobiles
  $\HM^1_g(\mu,\nu)\subset\HM_g(\mu,\nu)$ and a mapping
  $\Pi:\HC_g(\mu,\nu)\to\HM_g(\mu,\nu)$ such that
\[
\HC_g(\mu,\nu)\mathop{\longrightarrow}^\equiv_\Pi\HM^1_g(\mu,\nu)
\textrm{ and }\sigma\circ\Pi=\Pi\circ\sigma'.
\]

\item Finally we identify $\Pi(C)$ as $\Phi(G)$: more precisely we
  show ({Thm.~\ref{thm:main}}) that upon setting
  $\HM^{1c}_g(\mu,\nu)=\Pi(\HC^c_g(\mu,\nu))$, the composition
  $\Pi\circ\Gamma$ gives $\Phi$ and
\[
\HG_g(\mu,\nu)\mathop{\longrightarrow}^\equiv_{\Phi=\Pi\circ\Gamma}
\HM^{1c}_g(\mu,\nu)_{/\sigma}
\]
\item In the planar case, $\HC^0_g(\mu,\nu)=\HC_g(\mu,\nu)$ and
  $\HM^{1c}_g(\mu,\nu)=\HM_g(\mu,\nu)$, so that Thm.~\ref{thm:main}
  implies Thm.~\ref{thm:main-planar}.
\end{itemize}

\subsection{Trees and cacti}
As already observed, each non-marked vertex
of a galaxy $G$ has at least one incoming geodesic edge. By definition
of Hurwitz galaxy, all vertices have in-degree 1 or 2, hence at most
two incoming geodesic edges.

\begin{figure}[t]
\begin{center}
\includegraphics[scale=.47]{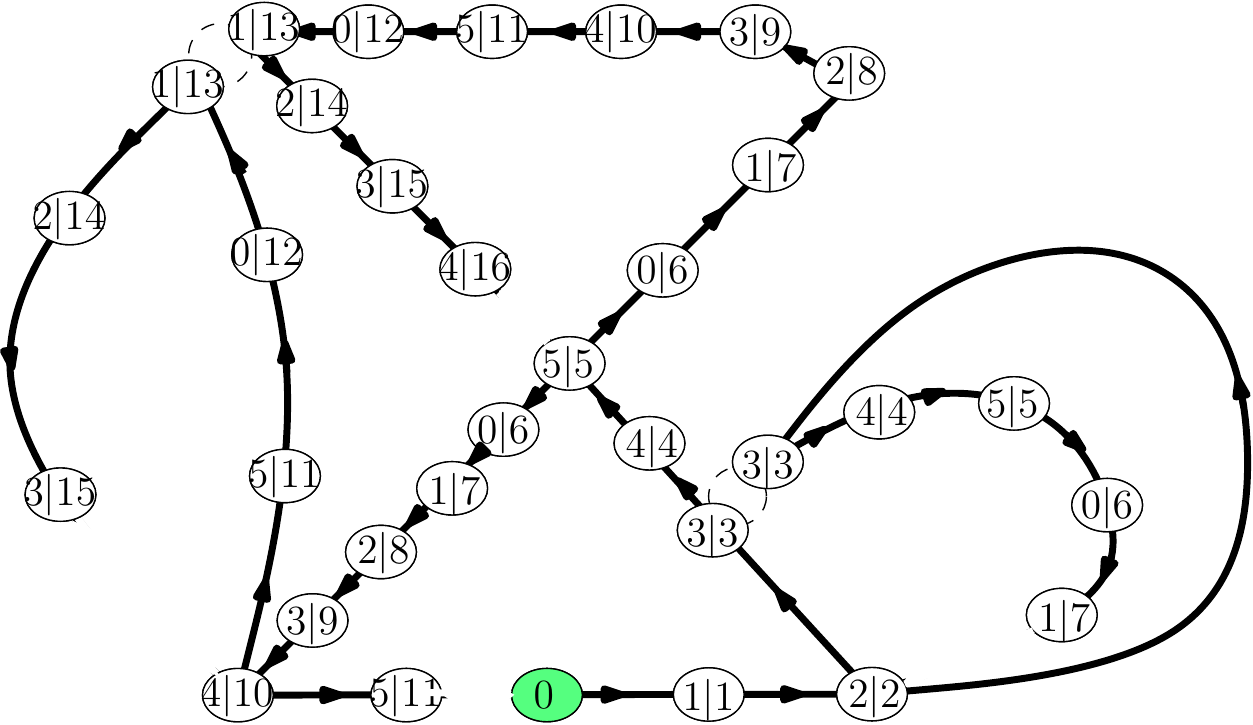}
\qquad
\includegraphics[scale=.47]{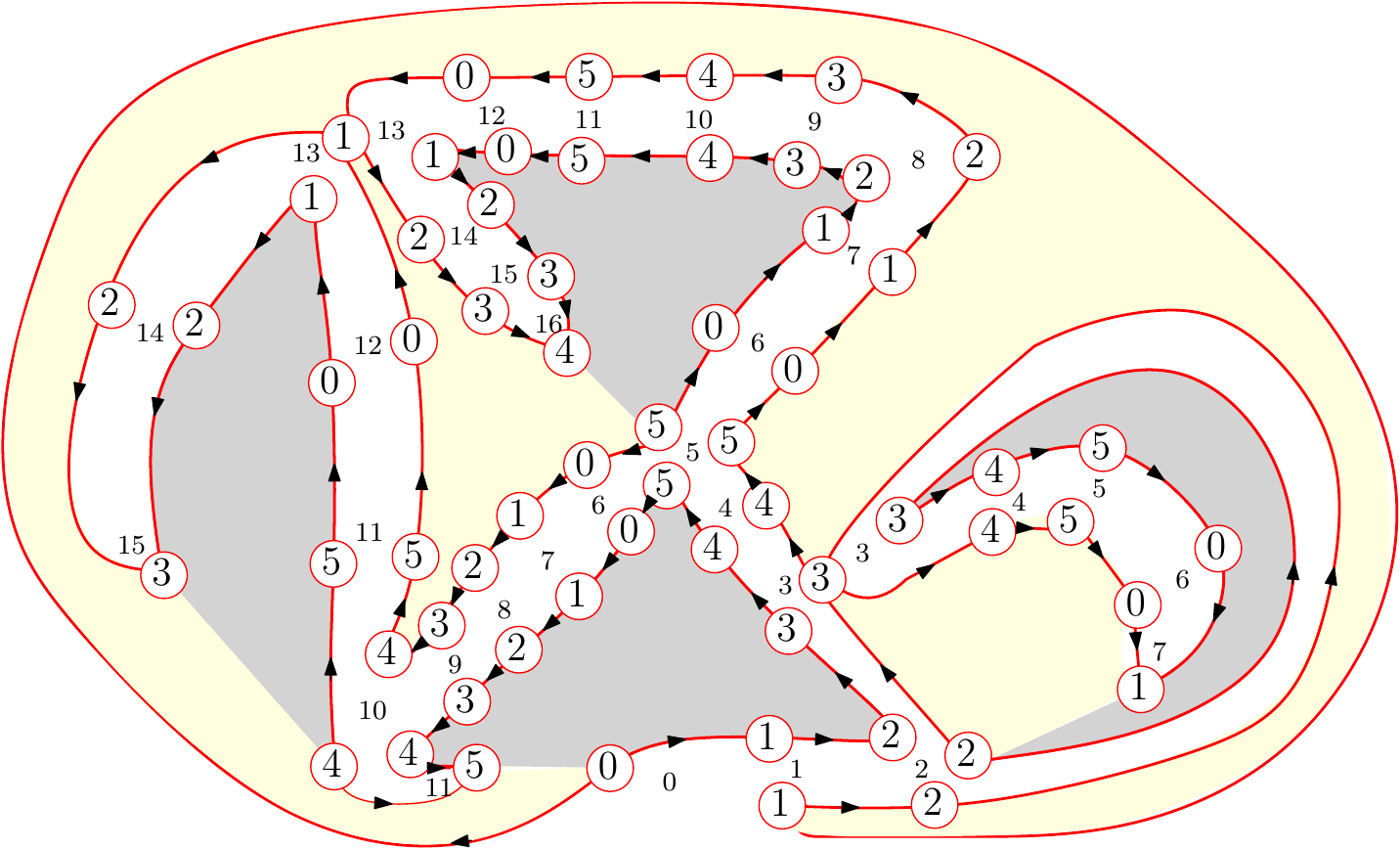}
\vspace{-.5em}
\end{center}
\caption{(a) The tree $\Theta(G)$ of geodesic edges of the galaxy of
  Fig.~\ref{fig:dist}(a) after vertex splitting.  (b) The colored
  cactus $\Gamma(G)$ obtained after cutting $\Theta(G)$ off $\S_g$:
  white faces are represented in yellow/lighter grey and the red curve
  is the boundary. Colors are indicated inside vertices, while the
  distance labels (or canonical corner labels) are written just
  outside. }\label{fig:split}
\end{figure}

\medskip
\noindent
\begin{minipage}{.5\linewidth}
The \emph{splitting} of a vertex $v$ with two incoming geodesic edges
consists in replacing $v$ by two new vertices, each carrying one
incoming geodesic edge and the outgoing edge following it in clockwise
direction around $v$. Let $\Theta(G)$ be the graph obtained by
splitting vertices with two incoming geodesic edges and removing
non-geodesic edges. Observe that the marked vertex $x_0$ has in-degree
1, so that it is not split and $x_0$ is a vertex of $\Theta(G)$.
\end{minipage}
\begin{minipage}{.49\linewidth}
\centerline{\includegraphics[scale=.6,page=1]{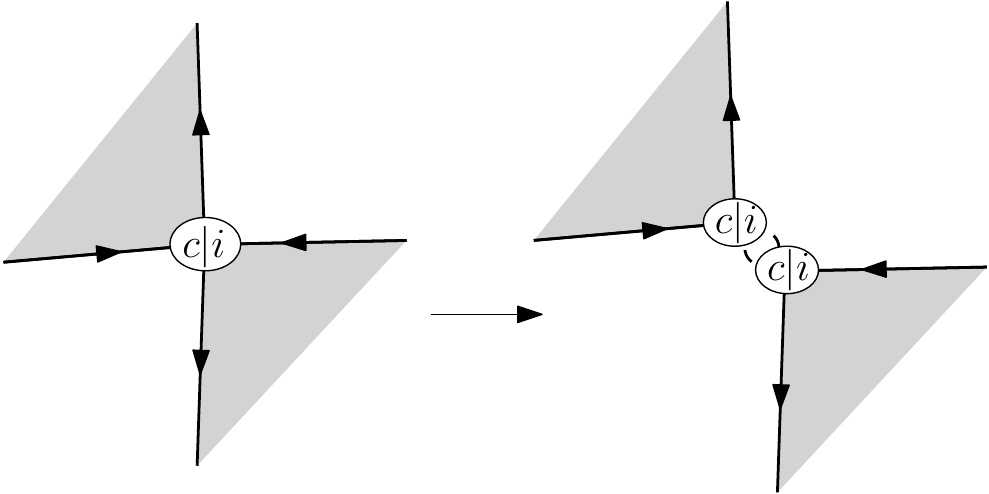}}
\end{minipage}

\begin{proposition}
The graph $\Theta(G)$ is a tree and for each vertex $v$ of $\Theta(G)$, $\delta(v)$
is given by the distance to the marked vertex $x_0$ in $\Theta(G)$.
\end{proposition}
\begin{proof}
By construction each vertex $v$ except the marked one has indegree one
in $\Theta(G)$. Moreover, if $v$ has label $\delta(v)=i$ ($i\geq 1$) the
edge arriving in $v$ in $\Theta(G)$ is a geodesic edge from $G$: in
particular it originates from a vertex $v'$ with label
$\delta(v')=i-1$. This implies by induction on $\delta(v)$ that all
vertices of $\Theta(G)$ are accessible from the marked vertex in this
graph. Hence $\Theta(G)$ is a tree and $\delta$ is the distance in
$\Theta(G)$. \hfill$\Box$
\end{proof}
In particular the distance labels can be recovered from the
(unlabeled) marked tree $\Theta(G)$.

Now assume that the galaxy $G$ is drawn on $\S_g$. Since $\Theta(G)$
is a tree, $\S_g\setminus \Theta(G)$ has one open boundary and its
closure is a surface $\S^\partial_g$ of genus $g$ with one boundary
(homeomorphic to a circle).  Let $\Gamma(G)$ be the map induced by $G$
on $\S^{\partial}_g$: By construction $\Gamma(G)$ directly inherits
from the faces and non-geodesic edges of $G$, while each geodesic edge
of $G$ produces two boundary edges in $\Gamma(G)$ (a white and a black
one, depending on the color of the incident face). The local analysis
of the possible configurations around each non marked vertex of $G$
yields the three cases presented in Figure~\ref{fig:local-cut} and
shows that each such vertex results in $\Gamma(G)$ into two or three
vertices, among which exactly one has some incoming white boundary
edges (the vertex represented by a square in each case of
Figure~\ref{fig:local-cut}). The marked vertex $x_0$ results in two
vertices without incoming edges.  Let us call \emph{active} the
vertices of $\Gamma(G)$ that have at least one incoming white boundary
edge. In particular each non marked vertex of $G$ corresponds to one
active vertex of $\Gamma(G)$.
\begin{figure}[t]
\centerline{
\includegraphics[scale=.4,page=4]{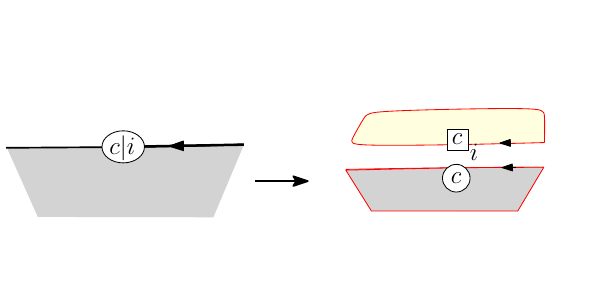}\qquad\qquad
\includegraphics[scale=.4,page=3]{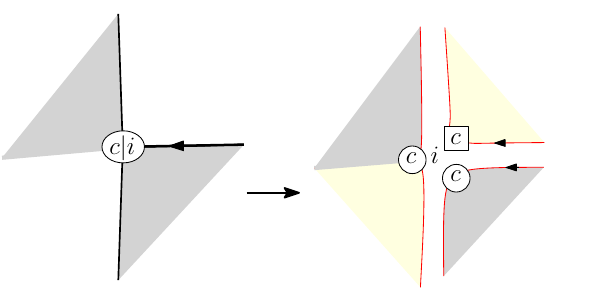}\qquad\qquad
\includegraphics[scale=.4,page=2]{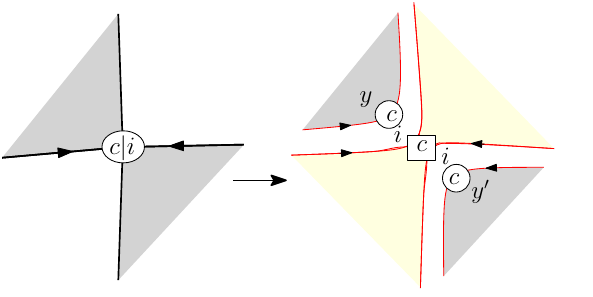}
\vspace{-.7em}
}
\caption{The three generic configurations in the construction of the
  induced map $\Gamma(G)$ on
  $\S^\partial_g=\overline{\S_g\setminus \Theta(G)}$:
  in-degree one; in-degree 2 with one geodesic incoming edge;
  in-degree 2 with two geodesic incoming edges. (In each case, the
  outgoing edges are represented as geodesic but could also be
  non-geodesic.)}\label{fig:local-cut}
\end{figure}
Let $\HC_g(\mu,\nu)$ denote the set of maps of genus $g$ with one
boundary such that:
\begin{itemize}
\item (Face color condition) There are $m_i$
  white faces of degree $(r+1)i$ and $n_i$ black faces of degree
  $(r+1)i$ for all $i$. There are three types of edges: \emph{internal
    edges} that are incident to a black and a white face; \emph{white
    boundary edges} that are oriented and have s white face on their
  right hand side; and \emph{black boundary edges} that have a black
  face on their left-hand side.
\item (Vertex color condition) All vertices are incident to the
  boundary, and have a color in $\{0,\ldots,r\}$ so that each
  (oriented) boundary edge $u\to v$ joins a vertex $u$ with color
  $c(u)$ to a vertex $v$ with color $c(v)=(c(u)\mod r+1)$.
\item (Hurwitz condition) There are $d-1$ active vertices of each
  color (recall that a vertex is {active} if it has at least one
  incoming white boundary edge).
\end{itemize}
The following lemma is then a rephrasing of the previous discussion.
\begin{lemma} The map $\Gamma(G)$ belongs to $\HC_g(\mu,\nu)$.
\end{lemma}
Given a map $C$ with one boundary and an orientation of the edges of
the boundary, a \emph{canonical corner labeling} is a mapping $\delta$
from the set of boundary corners of $C$ into non-negative integers
such that \emph{(a)} the minimum label is $0$, \emph{(b)} for each
boundary edge $e=u\to v$, $\delta(c')=\delta(c)+1$, where $c$
(resp. $c'$) is the boundary corner incident to $e$ at $u$ (resp. at
$v$).
In particular for any
galaxy $G$, the corner labeling of $\Gamma(G)$ inherited from the
distance labeling on $G$ is canonical by construction. This
construction is illustrated by Figure~\ref{fig:split}(b) (in the
picture corner labels common to nearby corners are shared to limit
cluttering).

\begin{lemma}
Each map $C\in\HC_g(\mu,\nu)$ has a unique canonical corner
labeling. 
\end{lemma}
\begin{proof}
Choose an arbitrary corner $c$ on the boundary of $C$ and give it
label 0. In view of Condition (b) above, the label of the next corner
in clockwise direction around the boundary is either $1$ or $-1$
depending if the traversed edge is a white or a black boundary
edge. All corner labels can be determined in this way and Condition
(b) is satisfied on the edge closing the boundary cycle because there
are equal numbers of black and white boundary edges (so that the
$\pm1$ walk giving labels automatically returns to zero when the
boundary has been entirely traversed). As a result all corners get
integer labels. Upon simultaneously shifting them so that the minimum
is 0, the lemma is proved. \hfill$\Box$
\end{proof}
The canonical corner labeling $\delta$ of a map $C\in\HC_g(\mu,\nu)$
is \emph{coherent} if for each vertex $u$ of $C$, all boundary corners
of $u$ have the same label. In this case the corner labeling yields a
vertex labeling called the \emph{coherent canonical labeling} of
$C$. The canonical labeling of Figure~\ref{fig:split}(b) is (by
construction) coherent, as can be checked around the two active
vertices of indegree 2, with labels 3 and 13 respectively. Let
$\HC^c_g(\mu,\nu)$ denote the set of maps of $\HC_g(\mu,\nu)$ whose
canonical corner labeling is coherent. In general not all maps of
$\HC_g(\mu,\nu)$ have a coherent canonical corner labeling, but this
is the case in the genus zero case:

\begin{proposition}\label{pro:planarcoherent}
Canonical corner labelings of maps in $\HC_0(\mu,\nu)$ are
coherent: $\HC_0(\mu,\nu)=\HC^c_0(\mu,\nu)$.
\end{proposition}
\begin{proof}
The only vertices of maps in $\HC_g(\mu,\nu)$ that are
incident to more than one boundary corner are the vertices gluing two
white polygons. In particular each such map decomposes as a collection
of components with simple boundaries glued by these vertices.

Now the maps in $\HC_0(\mu,\nu)$ are planar and have only one
boundary: their polygons thus form a tree-like structure (a kind of
\emph{cactus}). In particular each such map contains at least one
component that is connected to the rest by only one vertex (a leaf
polygon), and its canonical corner labeling is coherent if and only if
the map in which this component is removed is. Upon pruning the map
iteratively, the canonical corner labeling is seen to be coherent
everywhere.\hfill$\Box$
\end{proof}

Finally a map of
$\HC^c_g(\mu,\nu)$ is \emph{proper} if the (common) color of its vertices with
canonical label 0 is 0, and we denote by $\HC^{0c}_g(\mu,\nu)$
the corresponding subset of $\HC^c_g(\mu,\nu)$.
\begin{proposition}\label{pro:firstbij}
The mapping $\Gamma$ is a bijection between the sets $\HG_g(\mu,\nu)$  and
${\HC}^{0c}_g(\mu,\nu)$.
\end{proposition}
This proposition is a direct consequence of the following two lemmas.
\begin{lemma}
The decomposition $G\to (\Theta(G),\Gamma(G))$ is injective.
\end{lemma}
\begin{proof}
The boundary of $\Gamma(G)$ forms a cycle with twice as many edges as there
are edges in $\Theta(G)$: upon matching the marked vertex of $\Theta(G)$ with
the marked vertex of $\Gamma(G)$ there is a unique way to glue $\Theta(G)$ on
the boundary of $\Gamma(G)$ and recover $G$. \hfill$\Box$
\end{proof}
\begin{lemma}
The canonical corner labels around the boundary of $\Gamma(G)$ encodes
the tree $\Theta(G)$.
\end{lemma}
\begin{proof}
The counterclockwise sequence of labels around the boundary of
$\Gamma(G)$ starting from the marked vertex $x_0$ is exactly the
standard contour code of the plane tree $\Theta(G)$
\cite[Chap. 5]{book:Stanley} (\emph{aka} Dyck code, or discrete
excursion encoding the tree).  \hfill $\Box$
\end{proof}

\begin{proof}[of Proposition~\ref{pro:firstbij}]
The two lemmas above show that the mapping ${\Gamma}$ is
injective. Given an element $C_0$ of $\HC^{0c}_g(\mu,\nu)$, its
sequence of canonical corner labels encodes a tree on which $C_0$ can
be glued to form a map $G$ of genus $g$. This map is a galaxy 
of type $(\mu,\nu)$ in view of the face color and vertex color
conditions on maps of $\HC_g(\mu,\nu)$.  The fact that $C_0$ is
coherent allows to reconstruct a vertex labeling which coincide with
the distance labeling on $G$. Finally the Hurwitz condition on maps of
$\HC_g(\mu,\nu)$ ensures that $G$ is a Hurwitz galaxy.
\hfill$\Box$
\end{proof}

The \emph{shift} of a map $C\in{\HC}^{c}_g(\mu,\nu)$ consists in
adding one modulo $r+1$ to all colors. Recall that a map
$C\in{\HC}^c_g(\mu,\nu)$ is {proper} (that is, it belongs to
$\HC^{0c}_g(\mu,\nu)$) if the color of vertices with minimal label is
0.
\begin{proposition}\label{pro:shift}
Each shift-equivalence class of maps of ${\HC}^c_g(\mu,\nu)$
contains $r+1$ distinct maps, exactly one of which has a proper
canonical corner labeling, that is, belongs to $\HC^{0c}_g(\mu,\nu)$.
\end{proposition}
\begin{proof}
Shifting changes by one the (common) color of all the vertices that carry
the minimum corner label: there are therefore at least $r+1$ distinct
maps in a shift equivalence class. Moreover after $r+1$ shifts 
one returns to the original coloring so that there are exactly $r+1$ maps
in each shift equivalence class, and exactly one of these maps has
minimum label vertices of color 0.\hfill$\Box$
\end{proof}
\begin{corollary}\label{cor:shift}
The mapping $\Gamma$ is a bijection between the set $\HG_g(\mu,\nu)$
of Hurwitz galaxies of genus $g$ and type $(\mu,\nu)$ and the set of shift
equivalence classes of maps of ${\HC}^c_g(\mu,\nu)$.
\end{corollary}

\subsection{Simplifying cacti to get Hurwitz mobiles}

The graph $\Phi(G)$ constructed from a galaxy $G$ can be seen as the
retractation of $\Gamma(G)$: More precisely, given a Hurwitz galaxy
$G$, observe that the rules of Figure~\ref{fig:HM}(a)
and~\ref{fig:HM}(b) can be equivalently applied to $\Gamma(G)$ instead
of $G$ to construct $\Phi(G)$. Indeed the non-geodesic edges of $G$ to
which the rule of Figure~\ref{fig:HM}(a) applies exactly correspond to the
internal edges of $\Gamma(G)$, while the vertices that are split by
the rule of Figure~\ref{fig:HM}(b) correspond to the vertices incident
to two white faces in $\Gamma(G)$.  The fact that
$\Gamma(G)\in{\HC}^c_g(\mu,\nu)$ directly implies that $\Phi(G)$, as a
graph, is an edge-labeled Hurwitz mobiles of type $(\mu,\nu)$ and
excess $2g$ (in particular the Hurwitz condition on maps of $\HC_g(\mu,\nu)$
implies that the edge labels in the retract are all distinct).  We
can thus define the retractation map $\Pi$ from ${\HC_g(\mu,\nu)}$ to
$\HM_g(\mu,\nu)$ by the rules of Figure~\ref{fig:HM}(a) and~(b).

For the retractation of $\Gamma(G)$ to be reversible, one should
however be able to recover the map structure: the cyclic order of
edges around the nodes of polygons is fixed. As opposed to this, we
have defined Hurwitz mobiles as graphs (that is, without specifying an
embedding).  However observe that the order of edges around nodes of
the polygons of the retractation is determined by the fact that on
each node of a white (resp. black) polygon, the edge labels are
increasing in clockwise (resp. counterclockwise) order.  Let us define
the \emph{canonical embedding} of a Hurwitz mobile as the unique
embedding in a closed compact surface induced by these local
conditions.  We say that a Hurwitz mobile with excess $2g$ has
\emph{genus $g$} if this canonical embedding is an embedding in
$\S_g$.  Of course a Hurwitz mobile with excess $0$ has always genus
$0$, but for $g\geq1$ Hurwitz mobiles with excess $2g$ may have a
genus smaller than $g$, (in which case their canonical embedding has
several faces). Let $\tilde{{H}}_g(\mu,\nu)$ be the subset of
${H}_g(\mu,\nu)$ consisting of Hurwitz mobiles that have genus $g$.

\begin{proposition}\label{pro:retract}
The retract $\Pi$ is a bijection between ${\HC}_g(\mu,\nu)$ and
the set of $\HM^1_g(\mu,\nu)$ of Hurwitz mobiles with excess $2g$
and genus $g$ and type $(\mu,\nu)$. Moreover the shifts on
$\HC_g(\mu,\nu)$ and on $\HM^1_g(\mu,\nu)$ are
equivalent operations: for any $C\in \HC_g(\mu,\nu)$,
$\Pi(\sigma(C))=\sigma(\Pi(C))$.
\end{proposition}
\begin{proof}
As already discussed $R$ is a mapping from ${\HC}_g(\mu,\nu)$ to
${\HM}^1_g(\mu,\nu)$. Conversely given a Hurwitz mobile $M$, one
associates to each $i$-gon of $M$ a face of degree $(r+1)i$ divided
into $r+1$ subregions (the interior of the polygon and $i$ subregions
with boundary edges $\to0\to1\to\ldots\to r\to$ associated to the $i$
nodes of the $i$-gon). Then there is a unique way to embed locally
each white (resp. black) polygon and its incident edges in the
associated white (resp. black) face so that
\begin{itemize}
\item White (resp. black) polygons are drawn clockwise
(resp. counterclockwise).
\item The labels of edges incident to a given node on a white
  (resp. black) polygon increase in clockwise (resp. counterclockwise)
  direction between the two arcs incident to this node.
\item Each zero weight half-edge with color $c'$ reaches a boundary
  vertex with color $c'$ incident to the same subregion as its origin.
\item Each non-zero weight half-edge with color $c'$ reaches the
  middle of an edge with colors $c\to c'$, with $c'=c+1\mod r+1$.
\end{itemize}
The resulting faces can then be coherently glued together according to
the edges of $M$, and the result is a element of
${\HC}^1_g(\mu,\nu)$ since all local conditions are satified.
Finally the shift operation on $\HM_g(\mu,\nu)$ is a direct
translation through the retract of the shift on
$\HC_g(\mu,\nu)$.  \hfill$\Box$
\end{proof}

To combine this proposition with Proposition~\ref{pro:firstbij} we
need a last definition: an edge-labeled Hurwitz mobile of genus $g$ is
\emph{coherent} if the canonical corner labeling of the associated map
of $\HC_g(\mu,\nu)$ is. Let $\HM^{1c}_g(\mu,\nu)$ denote the set
of coherent edge-labeled Hurwitz mobiles of excess $2g$ and genus $g$ and type
$(\mu,\nu)$.  According to the
previous remarks, we have finally proved the following theorem.
\begin{theorem}\label{thm:main}
The mapping $M$ is a bijection between Hurwitz galaxies and
shift-equivalence classes of coherent edge-labeled Hurwitz mobiles
with the same genus and type. As a consequence, Hurwitz numbers of
genus $g$ count shift-equivalence classes of coherent Hurwitz mobiles
of genus $g$:
\[
h^\bullet_g(\mu,\nu)=\frac1{m+n-1+2g}\;|{\HM}^{1c}_g(\mu,\nu)|
\]
\end{theorem}
Proposition~\ref{pro:planarcoherent} shows that
$\HM^{1c}_0(\mu,\nu)=\HM_0(\mu,\nu)$, so that Theorem~\ref{thm:main}
implies in particular Theorem~\ref{thm:main-planar}.

\newpage
\section{Concluding remarks}
\mbox{}\indent 1) The mapping $\Gamma$ that we use in the first step
of our proof can be viewed as a reformulation of a special case of the
Bouttier-Di Francesco-Guitter construction \cite{BDFG04} in terms of
vertex splitting \cite{Bernardi-Fusy,Bernardi-Chapuy}. Our main
contribution here is to identify the image of $\Gamma$ and show that
it can be mapped (through $\Pi$) onto a set of well characterized
Hurwitz mobiles.

2) The construction in fact holds more generally for non-Hurwitz
galaxies where instead of $r$ simple branched points (with $d-1$
preimages), one requires $s$ branched points with respectively $d-r_i$
preimages with $\sum_{i=1}^sr_i=r$.  The resulting mobiles are
slightly more complex but again in the planar case explicit formulas
can be found for the corresponding \emph{double $s$-eulerian numbers}
in the terminology of Goulden and Jackson \cite{goulden-KP}.

3) The specialization of the bijection to the case $\mu=\nu=1^n$ is
already non trivial: let us call \emph{simple coverings of size $n$}
the corresponding branched coverings of the sphere by itself with
$2n-2$ critical values that are all simple, and \emph{simple galaxies}
the corresponding galaxies. Theorem~\ref{thm:main-planar} gives a
bijection between simple galaxies of size $n$ and a variant of Cayley
trees, namely edge-labeled Cayley trees with exactly one leaf incident
to each inner vertex (from Cayley's formula one immediately deduce
that these trees are counted by the formula
$n^{n-3}(2n-2)!/(n-1)!$). In view of the construction of the
bijection, the oriented pseudo-distance in the galaxy representing a
covering can be read on a canonical embedding in $\mathbb{Z}$ of the
associated tree induced by its the canonical corner labeling.  The
variation between successive corner labels associated to two
successive edges of the tree along a branch are easily seen to be
symetric random non zero integer variables taken uniformly in the
interval $[-n+1,n-1]$.  One can thus expect, in analogy with the many
results of convergence of embedded trees to the Brownian snakes
\cite{miermont08,marckert} that upon scaling the embedding support by
a factor $n^{-5/4}$, the resulting random embedded trees converge to
the Brownian snake. As a consequence, we expect that refinements of
the technics of \cite{miermont08,marckert} allow to prove that:
\begin{quote}\em
the length $L_n$ of the shortest oriented path between two random
vertices of uniform random simple galaxy of size $n$ satisfies
$L_n\cdot n^{-5/4}\to cte\cdot \mathcal{L}$ where $\mathcal{L}$ is the
distance between two random points in the Brownian map (see e.g.
\cite{chassaing2004random}).
\end{quote}

We then conjecture that, as the size tends to infinity, the $n^{-5/4}$
rescaled oriented pseudo-distances between the various pairs of points
of a same random simple galaxy become asymptotically symmetric, and
coincide a.s. with the $n^{-1/4}$ rescaled non-oriented distances on
the same galaxy up to a constant (non random) stretch factor. We
currently have no idea on how to prove such a result, but it would
imply that $n^{-1/4}$ rescaled uniform random simple galaxies of size
$n$ converge to the Brownian map in the same sense as uniform random
quadrangulations do
\cite{gall2011uniqueness,miermont2011brownian}. This would be in
agreement with the general (somewhat vague) assertion that uniform
random branched coverings of the sphere fall in the same universality
class as uniform random planar quadrangulations, and that they both
are natural discrete models of pure two-dimensional quantum geometries
(see \emph{e.g.} \cite{zvonkine2005enumeration}).

4) The graph metric defined by uniform random simple galaxies of size
$n$ is a particular set of random discrete metric space associated to
uniform random simple coverings of size $n$. As mentioned at the end of
Section~\ref{sec:galax}, there are other possible choices of curves on
the image sphere whose preimage yield different families of maps in
bijection with simple coverings: an example is given by the increasing
quadrangulations that arise as preimages of a bundle of parallel
edges separating the critical values \cite{DuPoSc14}.  We believe that
the uniform distribution on any such resulting family of maps induces,
upon considering the graph metric on it, a family of random discrete
metric spaces that also converges upon scaling to the Brownian map. A
very appealing approach to such a universality result would be to find
a way to relate these discrete metrics to the complex structures of
the underlying simple coverings.

\smallskip
\noindent\textbf{Acknowledgements.}\quad
M. Kazarian and D. Zvonkine are warmly thanked for sharing their
conjectures on double Hurwitz numbers and for interesting discussions.
The authors acknowledge support of ERC under grant StG 208471.

\newpage
\bibliographystyle{plain} \bibliography{HW}

\begin{thebibliography}{10}

\bibitem{Albenque-Poulalhon}
M.~Albenque and D.~Poulalhon.
\newblock A generic method for bijections between blossoming trees and planar
  maps.
\newblock {\em Electr. J. Combinatorics}, 2014.

\bibitem{Bernardi-Chapuy}
O.~Bernardi and C.~Chapuy.
\newblock A bijection for covered maps, or a shortcut between harer-zagier's
  and jackson's formulas.
\newblock {\em Journal of Combinatorial Theory - Series A}, 118(6):1718, 1748
  2011.

\bibitem{Bernardi-Fusy}
O.~Bernardi and E.~Fusy.
\newblock Unified bijections for maps with prescribed degrees and girth.
\newblock {\em Journal of Combinatorial Theory - Series A}, 119(6):1351--1387,
  2012.

\bibitem{borot2011matrix}
G.~Borot, B.~Eynard, M.~Mulase, and B.~Safnuk.
\newblock A matrix model for simple {H}urwitz numbers, and topological
  recursion.
\newblock {\em Journal of Geometry and Physics}, 61(2):522--540, 2011.

\bibitem{bousquet2000enumeration}
M.~Bousquet-M{\'e}lou and G.~Schaeffer.
\newblock Enumeration of planar constellations.
\newblock {\em Advances in Applied Mathematics}, 24(4):337--368, 2000.

\bibitem{BDFG04}
J.~Bouttier, P.~Di~Francesco, and E.~Guitter.
\newblock Planar maps as labeled mobiles.
\newblock {\em Electron. J. Combin.}, 11(1):\# 69, 2004.

\bibitem{cavalieri10}
R.~Cavalieri, P.~Johnson, and H.~Markwig.
\newblock Tropical hurwitz numbers.
\newblock {\em JACO}, 2010.

\bibitem{chassaing2004random}
P.~Chassaing and G.~Schaeffer.
\newblock Random planar lattices and integrated superbrownian excursion.
\newblock {\em Probability Theory and Related Fields}, 128(2):161--212, 2004.

\bibitem{DuPoSc14}
E.~Duchi, D.~Poulalhon, and G.~Schaeffer.
\newblock {\em Uniform random sampling of simple branched coverings of the
  sphere by itself}, chapter~21, pages 294--304.
\newblock 2014.

\bibitem{dunin-barkowski13}
P.~Dunin-Barkowski, M.~Kazarian, N.~Orantin, S.~Shadrin, and L.~Spitz.
\newblock Polynomiality of hurwitz numbers, bouchard-mariño conjecture, and a
  new proof of the elsv formula.
\newblock arXiv:1307.4729, 2013.

\bibitem{ekedahl2001hurwitz}
T.~Ekedahl, S.~Lando, M.~Shapiro, and A.~Vainshtein.
\newblock Hurwitz numbers and intersections on moduli spaces of curves.
\newblock {\em Inventiones mathematicae}, 146(2):297--327, 2001.

\bibitem{topo-rec}
B.~Eynard.
\newblock {\em Counting Surfaces: Combinatorics, Matrix Models, Algebraic
  Geometry}.
\newblock Birkh\"auser, 2014.

\bibitem{goulden1997transitive}
I.~Goulden and D.~Jackson.
\newblock Transitive factorisations into transpositions and holomorphic
  mappings on the sphere.
\newblock {\em Proceedings of the American Mathematical Society},
  125(1):51--60, 1997.

\bibitem{goulden-KP}
I.P. Goulden and Jackson D.M.
\newblock The {KP} hierarchy, branched covers, and triangulations.
\newblock {\em Advances in Mathematics}, 219:932--951, 2008.

\bibitem{Goulden-Jackson-Vakil}
I.P. Goulden, D.M. Jackson, and R.~Vakil.
\newblock Towards the geometry of double hurwitz numbers.
\newblock {\em Adv. Math.}, 198(1):43--92, 2005.

\bibitem{Johnson-ribbon}
P.~Johnson.
\newblock Hurwitz numbers, ribbon graphs, and tropicalization.
\newblock In {\em Tropical geometry and integral systems}, number 580 in
  Comtemp. Math., pages 55--72. Amer. Math. Soc., Providence, RI, 2012.
\newblock \texttt{arXiv:1303.1543}.

\bibitem{Johnson}
P.~Johnson.
\newblock Double hurwitz numbers via the infinite wedge.
\newblock {\em Transactions of the AMS}, 2014.
\newblock to appear, \texttt{arXiv:1008.3266}.

\bibitem{Kazarian}
M.~Kazarian.
\newblock Private communication.
\newblock 2014.

\bibitem{Kazarian-Lando}
M.~Kazarian and S.~Lando.
\newblock An algebro-geometric proof of witten's conjecture.
\newblock {\em J. Amer. Math. Soc.}, 20(4):1079--1089, 2007.

\bibitem{lando-zvonkin}
S.K. Lando and A.~K. Zvonkin.
\newblock {\em Graphs on Surfaces and their applications}.
\newblock Encyclopaedia of Mathematical Sciences. Springer-Verlag, 2004.

\bibitem{gall2011uniqueness}
J.-F. Le~Gall.
\newblock Uniqueness and universality of the brownian map.
\newblock {\em arXiv preprint arXiv:1105.4842, to appear in \emph{Ann.
  Proba.}}, 2011.

\bibitem{marckert}
J.-F. Marckert and G.~Miermont.
\newblock Invariance principles for random bipartite planar maps.
\newblock {\em The Annals of Probability}, 35(5):1642--1705, 2007.

\bibitem{miermont08}
G.~Miermont.
\newblock Invariance principles for spatial multitype galton-watson trees.
\newblock {\em Ann. Inst. Henri Poincaré (B)}, 44:1128--1161, 2008.

\bibitem{miermont2011brownian}
G.~Miermont.
\newblock The brownian map is the scaling limit of uniform random plane
  quadrangulations.
\newblock {\em arXiv preprint arXiv:1104.1606, to appear in \emph{Acta Math.}},
  2011.

\bibitem{natanzon}
S.~M. Natanzon and A.~V. Zabrodin.
\newblock Toda hierarchy, hurwitz numbers and conformal dynamics.
\newblock \texttt{arXiv\string:1302.7288}.

\bibitem{okounkov2000toda}
A.~Okounkov.
\newblock Toda equations for {H}urwitz numbers.
\newblock {\em Math. Res. Lett.}, 7(4):447--453, 2000.

\bibitem{okounkov2001gromov}
A.~Okounkov and R.~Pandharipande.
\newblock Gromov-{W}itten theory, {H}urwitz numbers, and matrix models, i.
\newblock {\em arXiv preprint math/0101147}, 2001.

\bibitem{PhD-Schaeffer}
G.~Schaeffer.
\newblock {\em Conjugaison d'arbres et cartes combinatoires}.
\newblock PhD thesis, Universit\'e Bordeaux 1, 1998.

\bibitem{SSV}
S.~Shadrin, M.~Shapiro, and A.~Vainshtein.
\newblock Chamber behavior of double hurwitz numbers in genus 0.
\newblock {\em Adv. Math.}, 217(1):79--96, 2008.

\bibitem{book:Stanley}
R.~P. Stanley.
\newblock {\em Enumerative combinatorics. {V}ol. 2}, volume~62 of {\em
  Cambridge Studies in Advanced Mathematics}.
\newblock Cambridge University Press, 1999.

\bibitem{zvonkine2005enumeration}
D.~Zvonkine.
\newblock Enumeration of ramified coverings of the sphere and 2-dimensional
  gravity.
\newblock {\em arXiv preprint math/0506248}, 2005.

\end{thebibliography}

\end{document}